\documentclass[11pt]{amsart}
\usepackage{amssymb,amsmath,amsthm}
\usepackage{verbatim}
\usepackage{graphicx}
\usepackage{epsfig}
\usepackage{color, soul}
\usepackage[all]{xy}

\DeclareMathOperator{\cb}{cb}
\DeclareMathOperator{\CB}{\mathcal{CB}}

\DeclareMathOperator{\proj}{proj}

\newcommand{\n}[1]{ \left\|#1\right\| }

\newcommand{\N}{{\mathbb{N}}}

\newtheorem{theorem}{Theorem}[section]
\newtheorem{lemma}[theorem]{Lemma}
\newtheorem{definition}[theorem]{Definition}
\newtheorem{corollary}[theorem]{Corollary}
\newtheorem{proposition}[theorem]{Proposition}
\newtheorem{remark}[theorem]{Remark}

\newtheorem{question}[theorem]{Question}
\newtheorem{problem}[theorem]{Problem}

\def\ker{{\rm ker\, }}

\begin{document}

\title[Operator $p$-compact mappings]{Operator $p$-compact mappings}

\author[J.A. Ch\'avez-Dom\'inguez]{Javier Alejandro Ch\'avez-Dom\'inguez}
\address{Department of Mathematics, University of Oklahoma, Norman, OK 73019-3103,
USA} \email{jachavezd@ou.edu}

\author{Ver\'onica Dimant}
\address{Departamento de Matem\'{a}tica y Ciencias, Universidad de San
Andr\'{e}s, Vito Dumas 284, (B1644BID) Victoria, Buenos Aires,
Argentina and CONICET} \email{vero@udesa.edu.ar}

\author{Daniel Galicer}
\address{Departamento de Matem\'{a}tica, Facultad de Ciencias Exactas y Naturales, Universidad de Buenos Aires, (1428) Buenos Aires,
Argentina and CONICET} \email{dgalicer@dm.uba.ar}

\thanks{
The first author was partially supported by NSF grant DMS-1400588.
The second author was partially supported by CONICET PIP 11220130100483 and ANPCyT PICT 2015-2299.  The third author was partially supported by CONICET PIP 11220130100329 and ANPCyT PICT 2015-3085.}

 \begin{abstract}
We introduce the class of operator $p$-compact mappings and completely right $p$-nuclear operators, which are natural extensions to the operator space framework of their corresponding Banach operator ideals. We relate these two classes, define  natural operator space structures and study several properties of these ideals.
We show that the class of operator $\infty$-compact mappings in fact coincides with a notion already introduced by Webster in the nineties (in a very different language). This allows us to provide an operator space structure to Webster's class.

 \end{abstract}

 \maketitle

\section*{Introduction}

Alexander Grothendieck, one of the most influential mathematicians of the 20th century, 
took his first steps in research in the area of functional analysis. In his early work, he developed a systematic theory of tensor products and norms \cite{grothendieck1956resume}, where he set the basis of what was later known as local theory --- that is, the study of Banach spaces in terms of their finite-dimensional subspaces --- and exhibited the importance of the use of tensor
products in the theory of Banach (and locally convex) spaces.
Grothendieck's approach, although at first it received little notice given its complexity and the language in which it was written, was rediscovered at the end of the sixties by Lindenstrauss and Pe{\l}czy{\'n}ski \cite{lindenstrauss1968absolutely} and inspired huge developments in Banach space theory.
In particular, it inspired the creation of a whole new theory: the study of operator ideals (systematized by Pietsch and the German school \cite{pietsch1968ideale,pietsch1978operator}).
In order to develop the metric theory of tensor products, Grothendieck gave a characterization of compactness in Banach spaces of independent interest \cite[Chap. I, p. 112]{grothendieck1955produits}: relatively compact sets are precisely those that lie within the absolutely convex hull of a null sequence. That is,
a set $K$ in a Banach space $\mathbf X$ is relatively compact if and only if there is a sequence $(x_n)_{n \in \mathbb{N}}$ in $\mathbf X$ such that

\begin{equation}\label{groth definicion de compacto}
K\subseteq \left\{\sum_{n=1}^{\infty} \alpha_n x_n \colon \sum_{n =1}^\infty \vert \alpha_n \vert \leq 1\right\} \mbox{ and } \lim_{n\to \infty} \Vert x_n \Vert = 0.
\end{equation}

Inspired by Grothendieck's result, Sinha and Karn \cite{sinha2002compact} introduced the notion of relatively $p$-compact
sets.
This definition describes relatively $p$-compact sets as those which are determined in a way similar to \eqref{groth definicion de compacto}, but by $p$-summable sequences instead of null ones.
That is, if $1 \leq p < \infty$ and $\frac{1}{p} + \frac{1}{p'} = 1$, a subset $K \subset \textbf X$ is called \emph{relatively $p$-compact} if there exists a sequence $(x_n)_n$ in  $\mathbf X$   such that
\begin{equation}\label{p-compacto}
K\subseteq \left\{\sum_{n=1}^{\infty} \alpha_n x_n \colon \sum_{n =1}^\infty \vert \alpha_n \vert^{p'} \leq 1\right\} \mbox{ and } \sum_{n=1}^{\infty} \Vert x_n \Vert^p < \infty.
\end{equation}
Of course, in the limiting case $p = 1$ we
replace the definition above by
$K\subseteq \left\{\sum_{n=1}^{\infty} \alpha_n x_n \colon \vert \alpha_n \vert \leq 1\right\}$  and $\sum_{n=1}^{\infty} \Vert x_n \Vert < \infty.$
Thus, classical compact sets can then be seen as ``$\infty$-compact''. Moreover, the following monotonicity relation holds: if $1 \leq q \leq p \leq \infty$, any relatively $q$-compact set is in fact relatively $p$-compact.
Therefore, $p$-compactness reveals ``finer and subtle'' structures on compact sets.

Having this definition at hand, it is natural to think about extending the notion of compactness to linear mappings.
By simple analogy, $p$-compact mappings arise as those which map the unit ball into a relatively $p$-compact set.
This definition, which comes from Sinha and Karn's original paper \cite{sinha2002compact}, turns out to yield a normed operator ideal, denoted by $\mathcal K_p$ (see the precise definition in Section \ref{section completely compact mappings}).
In recent years there has been great interest in this class --- and in some notions that naturally relate to it --- from a variety of perspectives including:
the theory of operator ideals \cite{prasad2008compact,delgado2010operators,delgado2010density,ain2012compact,pietsch2014ideal,turco2016mathcal,fourie2018injective}, the theory of tensor norms \cite{galicer2011ideal}, the interplay with various approximation properties \cite{delgado2009p,choi2010dual,lassalle2012p,lassalle2013banach,lassalle2016weaker,lassalle2017null,oja2012remark,oja2012grothendieck}, structural properties of sets and sequences \cite{pineiro2011p,ain2015p,kim2014unconditionally,fourie2017weak}, infinite dimensional complex analysis \cite{aron2010p,aron2011p,aron2012ideals,munoz2015alpha,aron2016behavior}.

The main goal of the present paper is to initiate the study of a noncommutative version of $p$-compact mappings, specifically in the context of operator spaces.
Recall that an operator space is a Banach space $\textbf X$ with an extra ``matricial norm structure'':
in addition to the norm on $\textbf X$, we have norms on all the spaces $M_n(\textbf X)$ of $n \times n$ matrices with entries from $\textbf X$ (and these norms must satisfy certain consistency requirements).
The morphisms are no longer just bounded maps, but the \emph{completely bounded ones}: they are required to be uniformly bounded on all the matricial levels.
Operator spaces are thus a quantized or noncommutative version of Banach spaces, giving rise to a theory that not only is mathematically attractive but it is also naturally well-positioned to have applications to quantum physics.
The systematic study of operator spaces started with Ruan's thesis, and
was developed mainly by Effros and Ruan, Blecher and Paulsen, Pisier and Junge (see the monographs \cite{junge1999factorization,Pisier-Asterisque-98,Blecher-LeMerdy-book,Effros-Ruan-book,Pisier-Operator-Space-Theory} and the references therein).

Due to their very definition, it is natural to investigate to what extent the classical theory of Banach spaces can be translated to the noncommutative context of operator spaces.
Though some properties do carry over, many do not and these differences are one of the reasons making the new theory so interesting.
In particular, both ideals of operators and tensor norms have inspired important developments in the operator space setting.
Noncommutative versions of nuclear, integral, summing, and other ideals of operators have played significant roles in, e.g.,
\cite{Effros-Ruan-book, Pisier-Asterisque-98,Junge-Habilitationschrift,Effros-Junge-Ruan,Junge-Parcet-Maurey-factorization},
whereas in addition to the usefulness of various specific tensor norms for operator spaces (most notably the Haagerup one),
Blecher and Paulsen \cite{Blecher-Paulsen-Tensors} have showed that the elementary theory of tensor norms of Banach spaces carries over to operator spaces, initiating a ``tensor norm program'' for operator spaces further developed in \cite{Blecher-1992}.
However, the abstract tool of associating ideals of operators and tensor norms in the sense of \cite[Chap. 17]{Defant-Floret} does not appear to have been developed yet.
In the present paper we take some small steps in this direction, which we need
throughout our definition and study of the new class of \emph{operator $p$-compact mappings} (which we denote by $\mathcal K_p^{o}$).
A much more thorough study of the relationships between tensor norms and ideals of mappings in the operator space category will appear in \cite{CDDG}.

One way to define this class is to rely on certain factorizations that characterize $p$-compact mappings via right $p$-nuclear maps. To this end, we study and define the  ideal of \emph{completely right $p$-nuclear mappings}, $\mathcal{N}^{p}_{o}$ (which are, in some way, a transposed version of the class given in \cite[Definition 3.1.3.1.]{Junge-Habilitationschrift} called $p$-nuclear operators).
In particular, $\mathcal K_p^{o}$ is the surjective hull of $\mathcal{N}^p_{o}$, as shown by Delgado, Pi\~neiro and Serrano \cite[Prop. 3.11]{delgado2010operators} and by Pietsch \cite[Thm. 1]{pietsch2014ideal} in the Banach space context.
This relation among these two classes enables us to give an appropriate operator space structure to $\mathcal K_p^{o}$.

At the end of the nineties, Webster (a student of Effros) proposed in his doctoral thesis  \cite{webster1997local} several ways to extend the notion of compactness to the operator space framework (see also \cite{Webster1998}).
One of these notions, which he called \emph{operator compactness}, is in effect a noncommutative version of Grothendieck's characterization  (being in the absolute convex hull of a null sequence as in Equation~\eqref{groth definicion de compacto}).
This also gave raise to the notion of \emph{operator compact mappings}.
Therefore, another path to translate the concept of $p$-compactness to the category of operator spaces is based on Webster's ideas.
We show that these two possible ways to define $p$-compactness in the operator space setting (the one which involves the factorization through completely right $p$-nuclear mappings and the one based on Webster's work) in fact coincide.
This allows us to endow Webster's class (of operator compact
mappings) with an operator space structure.

The  article is organized as follows.
In Section~\ref{Preliminaries} we present the notation and some basic concepts in the theory of operator spaces.
In Section~\ref{section right $p$-nuclear} we introduce the notion of completely right $p$-nuclear mappings, first in tensor terms and then proving a characterization in terms of factorizations. It should be mentioned that for $p=1$, this concept coincides with the completely nuclear mappings already studied in \cite[Section 12.2]{Effros-Ruan-book}.
In Section~\ref{section completely compact mappings} we introduce a notion of operator $p$-compact mappings based on factorization schemes, and give an operator space structure to this class by relating it to the aforementioned right $p$-nuclear mappings.
Additionally, we characterize operator $p$-compact mappings in terms of a noncommutative notion of $p$-compact sets.
We end the article with some open questions.

\section{Preliminaries} \label{Preliminaries}

Our Banach space notation is quite standard, and follows closely that of \cite{Defant-Floret,Ryan}. The injective tensor product of Banach spaces will be denoted by $\otimes_\varepsilon$.
For a Banach space $\mathbf{X}$ and $1\le p < \infty$, we denote by $\ell_p(\mathbf{X})$ and $\ell_{p}^w(\textbf X)$ the spaces of $p$-summable and weakly $p$-summable sequences in $\mathbf{X}$, respectively, with their usual norms
$$
  \Vert (x_n) \Vert_{\ell_{p}(\textbf X)} = \bigg( \sum_{n=1}^\infty \n{x_n}^p \bigg)^{1/p} , \qquad
 \Vert (x_n) \Vert_{\ell_{p}^w(\textbf X)} = \sup_{\n{x'} \le 1} \bigg( \sum_{n=1}^\infty  \big|x'(x_n)\big|^p \bigg)^{1/p}
$$
where $x'$ is taken in $\mathbf{X}'$, the dual of $\mathbf{X}$. The obvious modifications are adapted to the case $p=\infty$.
The unit ball of $\mathbf{X}$ is denoted by $B_{\mathbf{X}}$.

We only assume familiarity with the basic theory of operator spaces; the books \cite{Pisier-Operator-Space-Theory} and \cite{Effros-Ruan-book} are excellent references. Our notation follows closely that from \cite{Pisier-Asterisque-98, Pisier-Operator-Space-Theory}, with one main exception: we denote the dual of a space $E$ by $E'$.

Throughout the article $E$ and $F$ will be operator spaces. For each $n$, $M_n(E)$ will denote the space of $n\times n$ matrices of elements of $E$.
We will just write $M_n$ when $E$ is the scalar field.
The spaces $M_{n \times m}(E)$ and $M_{n\times m}$ are defined analogously.
For a linear mapping $T:E\to F$, we will denote its $n$-amplification by $T_n:M_n(E)\to M_n(F)$. The space of completely bounded linear mappings from $E$ into $F$ will be noted by $\CB (E,F)$.
The corresponding bilinear notion is the following: a bilinear map $T : E_1 \times E_2 \to F$ is jointly completely bounded if there exists a constant $C \ge 0$ such that for any matrices $(x_{i,j})_{i,j =1}^n \in M_n(E_1)$ and $(y_{k,\ell})_{k,\ell =1}^m \in M_m(E_2)$ we have
$$
\n{ \big(T(x_{i,j},y_{k,\ell})) }_{M_{nm}(F)} \le \n{ (x_{i,j})_{i,j =1}^n }_{M_n(E_1)}  \n{(y_{k,\ell})_{k,\ell =1}^m}_{M_m(E_2)}.
$$
The least such $C$ is denoted $\n{T}_{jcb}$.
As in \cite[Sec. 10.1]{Effros-Ruan-book} we will denote by $M_{\infty}(E)$ the space of all infinite matrices $(x_{i,j})$ with coefficients in $E$ such that the truncated matrices are uniformly bounded, i.e., $\sup_{n \in \N} \Vert (x_{i,j})_{i,j =1}^n \Vert_{M_n(E)} < \infty$. This supremum corresponds to the norm of the element $(x_{i,j})$ in $M_{\infty}(E)$. The space $M_\infty$ of bounded scalar infinite matrices is naturally identified with $\mathcal B(\ell_2)$.

Our notation for the minimal and the projective operator space tensor products will be, respectively, $\otimes_{\min}$ and $\otimes_{\proj}$.
By an operator space cross norm $\alpha$ we mean an assignment of an operator space $E \otimes_\alpha F$ to each pair $(E,F)$ of operator spaces, in such a way that $E \otimes_\alpha F$ is the algebraic tensor product $E \otimes F$ together with a matricial norm structure on $E \otimes F$ that we write as $\alpha_n$ or $\n{\cdot}_{\alpha_n}$, and such that
$\alpha_{nm}(x \otimes y) = \n{x}_{M_n(E)} \cdot \n{y}_{M_m(F)}$ for every $x \in M_n(E), y \in M_m(F)$.
This implies that the identity map $E \otimes_{\proj} F \to E \otimes_{\alpha} F$ is completely contractive \cite[Thm. 5.5]{Blecher-Paulsen-Tensors}.
If in addition the identity map $E \otimes_\alpha F \to E \otimes_{\min} F$ is also completely contractive, we say that $\alpha$ is an \emph{operator space tensor norm}.
Moreover, an operator space tensor norm $\alpha$ is called \emph{uniform} if additionally
for any operator spaces $E_1$, $E_2$, $F_1$, $F_2$, the map

$\otimes_\alpha : \CB(E_1,E_2) \times \CB(F_1,F_2) \to \CB(E_1 \otimes_\alpha E_2, F_1 \otimes_\alpha F_2)$
given by $(S,T) \mapsto S \otimes_\alpha T$
is jointly completely contractive.

If $\alpha$ is an operator space tensor norm, the completion of the tensor product $E \otimes_\alpha F$ will be denoted by $E \widehat{\otimes}_\alpha F$.
A degree of caution is required when consulting different works dealing with operator space tensor products, since the term ``tensor norm'' is not always taken to have the exact same meaning.
We point out that our definitions are slightly different than those of  \cite{Blecher-Paulsen-Tensors,Effros-Ruan-book,dimant2015biduals}, though not in any significant way.
An operator space tensor norm $\alpha$ is said to be \emph{finitely generated} if
for any operator spaces $E$ and $F$ and $u \in M_n(E \otimes F)$,
${\alpha}_n(u ; E, F) = \inf\left\{\alpha_n(u; E_0,F_0) \colon  u \in M_n(E_0 \otimes F_0)\right\} $,
where the infimum is taken over finite-dimensional subspaces $E_0 \subset E$ and $F_0 \subset F$.
For operator space tensor norms $\alpha$ and $\beta$ and a constant $c$, we write ``$\alpha \le c \beta$ on $E\otimes F$'' to indicate that the identity map $E \otimes_\beta F \to E \otimes_\alpha F$ has $\cb$-norm at most $c$.
If no reference to spaces is made, we mean that the inequality holds for any pair of operator spaces.
A linear map $Q : E \to F$ between operator spaces is called a complete 1-quotient  if it is onto and the associated map from $E/\ker(Q)$ to $F$ is a completely isometric isomorphism.
These maps are called complete metric surjections  in \cite[Sec. 2.4]{Pisier-Operator-Space-Theory}, where it is proved that a linear map $u:E \to F$ is a complete 1-quotient if and only if its adjoint $u' : F' \to E'$ is a completely isometric embedding.

For an operator space $E$ and $1\le p \le \infty$, let us define the spaces $S_p$ and $S_p[E]$ following \cite{Pisier-Asterisque-98}.
For $1 \le p <\infty$, $S_p$ denotes the space of Schatten $p$-class operators on $\ell_2$.
In the case $p=\infty$, we denote by $S_\infty$ the space of all compact operators on $\ell_2$.
We endow $S_\infty$ with the operator space structure inherited from $B(\ell_2)$, and $S_1$ with the one induced by the duality $S_1'= B(\ell_2)$; this then determines an operator space structure on $S_p$ for each $1<p < \infty$ via complex interpolation.
More generally, we define $S_\infty[E]$ as the minimal operator space tensor product of $S_\infty$ and $E$,
and $S_1[E]$ as the operator space projective tensor product of $S_1$ and $E$.
In the case $1 < p < \infty$, $S_p[E]$ is defined via complex interpolation between $S_\infty[E]$ and $S_1[E]$.
For $1 < p \le \infty$, the dual of $S_p[E]$ can be canonically identified with $S_{p'}[E']$, where $p'$ satisfies $1/p +  1/p' = 1$.

Recall that, given Banach spaces $\mathbf X$ and $\mathbf Y$ and $1 \leq p \leq \infty$, the $p$-right Chevet-Saphar tensor norm (\cite[Chapter 6]{Ryan}, \cite[Chapter 12]{Defant-Floret}) $d_p$ of a tensor $u \in \mathbf X \otimes \mathbf Y$ is defined by
\begin{equation}\label{dp-Banach}
d_p(u) = \inf \left\{\Vert (x_j) \Vert_{\ell_{p'}\otimes_{\varepsilon} X } \Vert(y_j)\Vert_{\ell_p(Y)} : u = \sum_{j} x_j \otimes y_j \right\},
\end{equation}
where the infimum runs over  all the possible ways in which the tensor $u$ can be written as a finite sum as above.

Moving to the operator space realm, for $1 \leq p \leq \infty$, in \cite{CD-Chevet-Saphar-OS} the $p$-right Chevet-Saphar operator space tensor norm $d_p^o$ is defined in the following way: given operator spaces $E$ and $F$, for $u \in E \otimes F$,
\begin{equation}\label{dp-os}
d_p^o(u) = \inf \left\{\Vert (x_{i,j}) \Vert_{S_{p'}\otimes_{min } E} \; \Vert(y_{i,j})\Vert_{S_p[F]} : u = \sum_{i,j} x_{i,j} \otimes y_{i,j} \right\},
\end{equation}
where the infimum runs over  all the possible ways in which the tensor $u$ can be written as a finite sum as above.

The operator space structure of the tensor $E \widehat \otimes_{d_p^o} F$  is given by the following 1-quotient  (see \cite[Section 3]{CD-Chevet-Saphar-OS})
 $$
 Q^p : (S_{p'} \widehat{\otimes}_{\min} E) \widehat{\otimes}_{\proj} S_p[F] \to E \widehat{\otimes}_{d^o_p} F,$$ where $\widehat \otimes $ means the completion of the corresponding tensor product.

An operator space $E$ is called projective if, for any $\varepsilon>0$, any completely bounded map $T:E \to F/S$ into a quotient space (here $F$ is any operator space and $S \subset F$ any closed space) admits a lifting $\widetilde T: E \to F$ with $\Vert \widetilde{T} \Vert_{\cb} \leq (1+ \varepsilon) \Vert T \Vert_{\cb}$.

Following \cite[Sec. 12.2]{Effros-Ruan-book}, a \emph{mapping ideal} $(\mathfrak{A},\mathbf{A})$ is an assignment, for each pair of operator spaces $E, F$, of a linear space $\mathfrak{A}(E,F) \subseteq \CB(E,F)$ together with an operator space matrix norm $\mathbf{A}$ on $\mathfrak{A}(E,F)$ such that
\begin{enumerate}
\item[(i)] The identity map $\mathfrak{A}(E,F) \to \CB(E,F)$ is a complete contraction.
\item[(ii)] The ideal property: whenever $T \in M_n(\mathfrak{A}(E,F))$, $r \in \CB(E_0,E)$ and $s \in \CB(F,F_0)$, it follows that
$
\mathbf{A}_n( s_n \circ T \circ r ) \le \n{s}_{\cb} \mathbf{A}_n(T) \n{r}_{\cb}.
$
\end{enumerate}

Finally, a remark about our use of the word \emph{operator}.
In the Banach space literature it is usual to speak of $p$-compact \emph{operators} rather than $p$-compact \emph{mappings}, and similarly for other various classes of bounded linear transformations between normed spaces.
In order to avoid confusions, throughout this paper we reserve the word \emph{operator} for notions that are noncommutative in nature (as in \emph{operator space}), and use the word \emph{mapping} to refer to bounded linear transformations.

 \section{Completely right $p$-nuclear mappings} \label{section right $p$-nuclear}

In this section we introduce the notion of completely right $p$-nuclear mappings, inspired by the  definitions in the Banach space setting.
For $p=1$, this concept coincides with the completely nuclear mappings studied in \cite[Section 12.2]{Effros-Ruan-book}.
Recall that in the Banach space setting a linear mapping $T: \mathbf X \to \mathbf Y$ is \emph{right $p$-nuclear} (\cite{reinov2000linear},\cite[Section 6.2]{Ryan}) if  $T$ can be written as
$$ T = \sum_{n=1}^\infty x_n' {\otimes} y_n,$$
where $(x_n') \in \ell_{p'}^w(\textbf X')$, $(y_n) \in \ell_{p}(\textbf Y)$. Moreover, the right $p$-nuclear norm of $T$ is defined as

$$ \nu^p(T):=\inf \{ \Vert (x_n) \Vert_{\ell_{p'}^w(\textbf X')} \cdot \Vert (y_n) \Vert_{\ell_{p}(\textbf Y)}\},$$
where the infimum is taken all over possible representations of $T$ as above. The class of all right $p$-nuclear mappings is denoted by $\mathcal N^p$.
This definition is known to be equivalent to having a factorization
\begin{equation} \label{right p-nuclear}
\xymatrix{
\mathbf X  \ar[r]^T \ar[d]_{U} &\mathbf Y \\
\ell_{p'} \ar[r]_{D_\lambda} &  \ell_1 \ar[u]_{V},
}
\end{equation}
where $U$ and $V$ are bounded mappings, $\lambda \in \ell_p$, and $D_\lambda$ stands for the diagonal multiplication mapping $(x_n) \mapsto (\lambda_nx_n)$.  Moreover,
$
\nu^p(T) = \inf \{ \Vert U \Vert  \Vert D_{\lambda} \Vert  \Vert V \Vert \}
$
where the infimum runs over all factorizations as above.
It is well-known that right $p$-nuclear mappings between the Banach spaces $\mathbf X$ and $\mathbf Y$ are exactly those mappings which are in the range of the canonical inclusion
$
 \mathbf  X' \widehat \otimes_{d_p}\mathbf  Y  \to \mathbf X' \widehat \otimes_{\varepsilon} \mathbf Y,
$
and the right $p$-nuclear norm coincides with the quotient norm inherited from this inclusion (see also the analogous results given in \cite[22.3]{Defant-Floret}, \cite[Section 6.2]{Ryan} and \cite[Proposition 5.23]{Diestel-Jarchow-Tonge}).

Motivated by this, we introduce a corresponding notion in the category of operator spaces.
\begin{definition}\label{p-nuclear def}
Let $1 \leq p \leq \infty$. We say that a linear mapping $T : E \to F$ between operator spaces $E$ and $F$ is \emph{completely right $p$-nuclear} if it corresponds to an element in the range of the canonical inclusion
$$
J^p : E' \widehat{\otimes}_{d_p^o} F \to E' \widehat{\otimes}_{\min} F.
$$
We denote the space of all such mappings by $\mathcal{N}^p_o(E,F)$,
and we endow it with the quotient structure $(E' \widehat{\otimes}_{d_p^o} F)/ \ker J^p$. In particular, its norm -- that we denote by $\nu^p_o$ -- is the quotient norm.
\end{definition}

For future reference, it is important to remark that since the operator space structure on $E' \widehat{\otimes}_{d_p^o} F$ is itself coming from a quotient,
the above definition is equivalent to being in the range of
\begin{equation*}
J^p \circ Q^p : (S_{p'} \widehat{\otimes}_{\min} E') \widehat{\otimes}_{\proj} S_p[F] \to E' \widehat{\otimes}_{\min} F,
\end{equation*}
and this still induces the same quotient structure.

Let us start by proving that $\mathcal{N}^p_o$ is a mapping ideal. We observe, also, that each space  $\mathcal N^p_o (E,F)$ is complete.

\begin{proposition}\label{proposition:ideal-property}
$\mathcal N^p_o$ is a mapping ideal, for $1 \leq p \leq \infty$.
\end{proposition}

\begin{proof}
Given $T \in M_n (\mathcal N_o^p (E;F))$, the definition of right $p$-nuclearity yields  $\Vert T \Vert_{M_n(\mathcal CB(E;F))} \leq \Vert T \Vert_{M_n(\mathcal N_o^p(E;F))}$.
Also for $S \in \CB(E_0,E)$ and $R \in \CB(F,F_0)$ we have $R_nTS \in M_n\left(\mathcal{N}^p_o(E_0,F_0)\right)$ and
$$
\Vert R_nTS \Vert_{M_n(\mathcal N^p_o(E_0,F_0))} \le \n{R}_{\cb} \Vert T \Vert_{M_n(\mathcal N^p_o(E;F))} \n{S}_{\cb}.
$$
Indeed, by definition there exists $t \in M_n(E' \widehat{\otimes}_{d_p^o} F)$ such that $(J^p)_n(t) = T$.
Since $d_p^o$ is a uniform operator space tensor norm,
$$
d_p^o\big( (S' \otimes R)_n(t) \big) \le \n{S'}_{\cb} d_p^o(t) \n{R}_{\cb} = \n{S}_{\cb} d_p^o(t) \n{R}_{\cb}.
$$
Note that $(J^p)_n \big( (S' \otimes R)_n(t)  \big) = R_nTS$, so $R_nTS \in M_n(\mathcal{N}^p_o(E_0,F_0))$ and
$$
\Vert R_nTS \Vert_{M_n(\mathcal N^p_o(E_0,F_0))} \le \n{S}_{\cb} d_p^o(t) \n{R}_{\cb}.
$$
Taking the infimum over all $t$ we get the desired conclusion.
\end{proof}

The following result shows that the formula given in Equation \eqref{dp-os} can be extended for tensors that lie in the completion $ E \widehat{\otimes}_{d_p^o} F$.

\begin{theorem}\label{thm:hard-Chevet-Saphar}
Let $1 \leq p \leq \infty$, let $E$ and $F$ be operator spaces, and $u \in E \widehat{\otimes}_{d_p^o} F$. Then
$$
\n{u}_{d_p^o} = \inf \big\{   \n{(x_{ij})}_{S_{p'}\widehat{\otimes}_{\min} E} \n{(y_{ij})}_{S_{p}[F]} \; : \; u = \sum_{i,j} x_{ij} \otimes y_{ij}    \big\}.
$$
\end{theorem}

\begin{proof} Let $u \in E \widehat{\otimes}_{d_p^o} F$.
It is clear that $\n{u}_{d_p^o} \le   \n{(x_{ij})}_{S_{p'}\widehat{\otimes}_{\min} E} \n{(y_{ij})}_{S_{p}[F]}$, for any representation of $u= \sum_{i,j} x_{ij} \otimes y_{ij}= Q^p\left(( x_{ij})\otimes( y_{ij})\right)$, since $Q^p$ is a 1-quotient mapping.

Also,  for every $\eta >0$ there is a sequence $(u_m) \in E \otimes F$ such that $u = \sum_{m} u_m \in E \widehat{\otimes}_{d_p^o} F$ with $d_p^o(u_1) < d_p^o(u) + \eta$ and
$ d_p^o(u_m) \leq \frac{\eta^2}{4^m}$ for every $ m \geq 2$. Now, we can write
$u_1 := \sum_{i,j =  1}^{k_1} x_{i,j} \otimes y_{i,j}$ where
$\n{(x_{ij})_{i,j=1}^{k_1}}_{S_{p'}^{k_1} \otimes_{\min} E} \leq d_p^o(u) + \eta$ and $\n{(y_{ij})_{i,j=1}^{k_1}}_{S_p[F]} \leq 1$. Also, for $m \geq 2$, we may represent $u_m : = \sum_{i,j = k_{m-1} + 1}^{k_m} x_{i,j} \otimes y_{i,j}$  with
$\n{(x_{ij})_{i,j={k_{m-1} + 1}}^{k_2}}_{S_{p'}^{k_2} \otimes_{\min} E} \leq \frac{\eta}{2^m}$ and $\n{(y_{ij})_{i,j={k_{m-1} + 1}}^{k_2}}_{S_p[F]} \leq \frac{\eta}{2^m}.$
By the triangle inequality we derive $\n{(y_{ij})}_{S_{p}[F]} \leq 1+ \eta \sum_{m=2}^\infty 2^{-m}$ and also
$\n{(x_{ij})}_{S_{p'}\widehat{\otimes}_{\min} E} \leq d_p^o(u) + \eta + \eta \sum_{m=2}^\infty 2^{-m}$, which concludes the proof.
\end{proof}

We now introduce a non-commutative version of the sequence space $\ell_p^w(E)$, namely

$$ S_{p}^w[E]:=\{ x \in M_{\infty}(E) :   \sup_{N} \Vert (x_{ij})_{i,j=1}^N  \Vert_{S_p^N \widehat{\otimes}_{\min} E } < \infty \}. $$

It can be easily seen that this is an operator space endowed with the matricial norm structure given by

$$ \Vert  \big (x_{ij}^{k,l})_{i,j} \big)_{k,l=1}^n  \Vert_{M_n(S_{p}^w[E])} := \sup_{N} \Vert \big((x_{ij}^{k,l})_{i,j=1}^N \big)_{k,l=1}^{n} \Vert_{M_n(S_p^N \widehat{\otimes}_{\min} E) }.$$

Recall that $\ell_p^w(E)$ can be identified with the space of bounded linear mappings from $\ell_{p'}$ to $E$ \cite[Prop. 8.2.(1)]{Defant-Floret}.
The following is the analogous statement for $S_{p}^w[E]$.

\begin{lemma}\label{lemma-Spweak-as-mappings}
For $1 \leq p \leq \infty$, we have the following completely isometric identification $$S_{p}^w[E] = \CB(S_{p'},E).$$
\end{lemma}

\begin{proof}
Let us see first that the spaces $S_{p}^w[E]$ and $\CB(S_{p'},E)$ are isometrically isomorphic. For each $N\in\mathbb N$, we denote $\rho_N:M_{\infty}(E)\to M_N(E)$ the $N$-truncation given by $\rho_N\left((x_{ij})_{ij}\right)=(x_{i,j})_{i,j=1}^N$. Each $X \in S_p^w[E]$ satisfies that
$$
\|X\|_{S_p^w[E]}=\sup_N \|\rho_N(X)\|_{S_p^N \widehat{\otimes}_{\min} E}=\sup_N \|\rho_N(X)\|_{\CB(S_{p'}^N,E)}
$$ since the spaces $S_p^N \widehat{\otimes}_{\min} E$ and $\CB(S_{p'}^N,E)$ are isometric. Thus, the mapping $T_X: \bigcup_{N=1}^\infty S_{p'}^N \to E$ is well defined (given $A \in S_{p'}^N$ for some $N \in \mathbb{N}$, $T_X(A):=\rho_N(X) (A)$) and uniquely extends to $S_{p'}$  by density.
The same holds with the $n$-amplification $(T_X)_n:M_n(\bigcup_{N=1}^\infty S_{p'}^N) \to M_n(E)$ (whose norm is obviously bounded by $\Vert X \Vert_{S_p^w[E]}$), and extends to $M_n(S_{p'})$. Hence, $T_X\in \CB(S_{p'},E)$ with $\|T_X\|_{\CB(S_{p'},E)}\le \Vert X \Vert_{S_p^w[E]}$.

Conversely, for $T\in \CB(S_{p'},E)$, denoting by $i_N:S_{p'}^N\to S_{p'}$ the canonical completely isometric inclusion, we derive that $T\circ i_N\in \CB(S_{p'}^N,E)$ with $\|T\circ i_N\|_{\CB(S_{p'}^N,E)}\le \|T\|_{\CB(S_{p'},E)}$, for each $N\in\mathbb N$. Now, defining $X_T\in M_\infty (E)$ by $X_T=\left(Te_{ij}\right)_{i,j}$ it is plain, through the identity $S_p^N \widehat{\otimes}_{\min} E=\CB(S_{p'}^N,E)$, that
$$
\|X_T\|_{S_p^w[E]}=\sup_N\|\rho_N(X_T)\|_{S_p^N \widehat{\otimes}_{\min} E}=\sup_N\|T\circ i_N\|_{\CB(S_{p'}^N,E)}\le  \|T\|_{\CB(S_{p'},E)}.
$$

Hence, we have proved the isometry $S_{p}^w[E] = \CB(S_{p'},E)$. To see that this isometry is  complete, we need to show that, for each $n$, the spaces $M_n\left(S_{p}^w[E]\right)$ and $M_n\left(\CB(S_{p'},E)\right)$ are isometric. Recall that
$$
M_n\left(\CB(S_{p'},E)\right) \overset{1}{=} \CB(S_{p'},M_n(E))  \overset{1}{=} S_{p}^w[M_n(E)],
$$ where the last equality follows from the first part of the proof. To finish it only remains to see that $ S_{p}^w[M_n(E)]\overset{1}{=} M_n\left(S_{p}^w[E]\right)$. Indeed, by \cite[Cor. 8.1.3 and Cor. 8.1.7]{Effros-Ruan-book}, for all $N$ the spaces $S_p^N\widehat{\otimes}_{\min} M_n(E)$ and $M_n\left(S_p^N \widehat{\otimes}_{\min} E\right)$ are completely isometric. Therefore, for each $X\in M_n\left(S_{p}^w[E]\right)$, it holds
$$
\|X\|_{M_n\left(S_{p}^w[E]\right)}=\sup_N \|(\rho_N)_n(X)\|_{M_n\left(S_p^N \widehat{\otimes}_{\min} E\right)} =\sup_N \|(\rho_N)_n(X)\|_{S_p^N\widehat{\otimes}_{\min} M_n(E)} = \|X\|_{S_{p}^w[M_n(E)]}.
$$

\end{proof}

As in the Banach space setting, we can replace $S_{p'}\widehat{\otimes}_{\min} E$ by $S_{p'}^w[E]$ in the quotient description  of the norm $d_p^o$.

\begin{theorem} \label{Thm:Q-tilde}
For $1 \leq p \leq \infty$, the mapping
$$
\widetilde{Q}^p: S_{p'}^w[E] \widehat{\otimes}_{\proj} S_p[F] \to E \widehat{\otimes}_{d^o_p} F,
$$ is a complete 1-quotient.
\end{theorem}

\begin{proof} We have just to prove that $\widetilde{Q}^p$ defines a complete contraction. The fact that $Q^p$ is a complete 1-quotient and the complete isometry $S_{p'}\widehat{\otimes}_{\min} E\hookrightarrow S_{p'}^w[E]$ do the rest.

Let us denote by $\tau^k: S_{p'}^w[E]\to S_{p'}^k\otimes_{\min} E$ the truncation mapping given by $\tau^k\left((x_{ij})_{i,j}\right)=(x_{ij})_{i,j=1}^k$. It is plain, due to the operator space structure of $S_{p'}^w[E] $, that $\tau^k$ is a complete contraction. We note by $\widetilde{\tau}^k$ the ``same''
 truncation mapping but on a different space, $\widetilde{\tau}^k:S_p[F]\to S_p[F]$.

Let $(x_{ij}) \in S_{p'}^w[E]$ and $(y_{ij}) \in S_{p}[F]$. To see that $\sum_{i,j} x_{ij} \otimes y_{ij}$ converges in  $E \widehat{\otimes}_{d_p^o} F$ we prove that $(u_k)$ is a Cauchy sequence, where $u_k = \sum_{i,j=1}^n  x_{ij} \otimes y_{ij}$.
By definition,
$$d_p^o(u_k - u_l ) \leq \n{\tau^k(x_{ij})-\tau^l(x_{ij})}_{S_{p'}^m \otimes_{\min} E} \n{\widetilde{\tau}^k(y_{ij})-\widetilde{\tau}^l(y_{ij})}_{S_p[F]}.$$
Note that $\n{\tau^k(x_{ij})-\tau^l(x_{ij})}_{S_{p'}^m\otimes_{\min} E} \leq 2\n{(x_{ij})}_{S_{p'}^w[E]} $
and by \cite[Lemma 1.12]{Pisier-Asterisque-98}, $\n{\widetilde{\tau}^k(y_{ij})-\widetilde{\tau}^l(y_{ij})}_{S_p[F]} \leq \varepsilon$ for $k,\ l$ sufficiently large.
So, $(u_k)_k$ converges to $u = \sum_{i,j} x_{ij} \otimes y_{ij} $ and
$d_p^o(u) \leq \n{(x_{ij})}_{S_{p'}^w[E]} \n{(y_{ij})}_{S_{p}[F]}$. Thus, $\widetilde{Q}^p$ is  well defined.

Now, let $n\in\mathbb N$ fixed. We need to prove that
$$
\left\|(\widetilde{Q}^p)_n: M_n\left(S_{p'}^w[E] \widehat{\otimes}_{\proj} S_p[F]\right)\to M_n\left( E \widehat{\otimes}_{d^o_p} F\right)\right\|\le 1.
$$

 Recall that, given $r\in\mathbb N$ and $Y\in M_r(S_p[F])$, by \cite[Lemma 1.12]{Pisier-Asterisque-98}, it holds
 $$
 \left\|(\widetilde\tau^k)_r(Y)-Y\right\|_{M_r(S_p[F])}\underset{k\to\infty}{\longrightarrow} 0.
 $$

 For $A\in M_n\left(S_{p'}^w[E] {\otimes}_{\proj} S_p[F]\right)$ take any representation of the form $A=\alpha (X\otimes Y)\beta$, with $X\in M_q\left(S_{p'}^w[E]\right)$, $Y\in M_r\left(S_p[F]\right)$, $\alpha\in M_{n,q\cdot r}$, $\beta\in M_{q\cdot r, n}$. By arguing as in the beginning of the proof we obtain
$$
 (Q^p)_n  (\tau^k\otimes\widetilde\tau^k)_n(A)=(\widetilde Q^p)_n  (\tau^k\otimes\widetilde\tau^k)_n(A) \underset{k\to\infty}{\longrightarrow} (\widetilde Q^p)_n (A), \textrm{ in } M_n (E \widehat{\otimes}_{d^o_p} F).
$$
 Also, the fact that $Q^p$ is a complete 1-quotient implies, for each $k$,
\begin{eqnarray*}
 \left\|(Q^p)_n  (\tau^k\otimes\widetilde\tau^k)_n(A)\right\|_{M_n (E \widehat{\otimes}_{d^o_p} F)} & = & \left\|(Q^p)_n \left(\alpha (\tau_q^k(X) \otimes \widetilde\tau^k_r(Y)\right)\beta\right\|_{M_n (E \widehat{\otimes}_{d^o_p} F)} \\
 &\le & \|\alpha\|_{M_{n,q\cdot r}} \cdot \|\tau_q^k(X) \|_{M_q(S_{p'}^k\otimes_{\min} E)} \cdot \|\widetilde\tau^k_r(Y)\|_{M_r(S_p[F])} \cdot \|\beta\|_{M_{q\cdot r, n}}\\
 &\le & \|\alpha\|_{M_{n,q\cdot r}} \cdot \|X \|_{M_q\left(S_{p'}^w[E]\right)} \cdot \|\widetilde\tau^k_r(Y)\|_{M_r(S_p[F])} \cdot \|\beta\|_{M_{q\cdot r, n}}.
\end{eqnarray*}
Taking limit as $k\to\infty$ we obtain
$$
\left\|(\widetilde Q^p)_n (A)\right\|_{M_n (E \widehat{\otimes}_{d^o_p} F)} \le \|\alpha\|_{M_{n,q\cdot r}} \cdot \|X \|_{M_q\left(S_{p'}^w[E]\right)} \cdot \|Y\|_{M_r(S_p[F])} \cdot \|\beta\|_{M_{q\cdot r, n}}.
$$
Since this holds for every representation of $A$ it is clear that
$$
\left\|(\widetilde Q^p)_n (A)\right\|_{M_n (E \widehat{\otimes}_{d^o_p} F)} \le \|A\|_{M_n\left(S_{p'}^w[E] {\otimes}_{\proj} S_p[F]\right)}.
$$
 \end{proof}

Two direct consequences of the previous theorem are stated in the next corollary:

\begin{corollary} \label{defB_p} Let $1 \leq p \leq \infty$ then:
\begin{enumerate}
\item The bilinear mapping associated to $\widetilde{Q}^p$,  $B^p: S_{p'}^w[E]\times S_p[F]\to E\widehat\otimes_{d_p^o} F$ given by
\begin{equation*}
B^p\left((x_{ij}), (y_{ij})\right)= \sum_{i,j} x_{ij} \otimes y_{ij}
\end{equation*} is jointly completely bounded with $\|B^p\|_{jcb}=1$.

\item For each $u \in E \widehat{\otimes}_{d_p^o} F$,
$$
\n{u}_{d_p^o} = \inf \big\{   \n{(x_{ij})}_{S_{p'}^w[E]} \n{(y_{ij})}_{S_{p}[F]} \; : \; u = \widetilde{Q}^p\left( (x_{ij})\otimes (y_{ij})\right)    \big\}.
$$
\end{enumerate}
\end{corollary}

Now we show that certain two-sided multiplication operators are completely right $p$-nuclear.
This is closely related to the results in \cite{Oikhberg-10}, and the argument here follows closely \cite[Prop. 12.2.2]{Effros-Ruan-book}.
As in \cite[p. 21]{Pisier-Asterisque-98}, if $E$ is an operator space, $x \in M_\infty(E)$ and $a \in M_\infty$, we denote by $a \cdot x$ (resp. $x \cdot a$) the matrix product, that is,
$$
(a \cdot x)_{ij} = \sum_k a_{ik}x_{kj}, \qquad \bigg( \text{resp. } \quad (x\cdot a)_{ij} = \sum_k x_{ik}a_{kj}  \bigg).
$$
If $b \in M_\infty$, we denote $a \cdot x \cdot b = a \cdot (x \cdot b) = (a \cdot x) \cdot b$.

\begin{proposition}\label{proposition:multiplication-operators-are-nuclear}
Let $1 \leq p \leq \infty$ and  $a,b \in S_{2p}$. Then the multiplication operator
$
M(a,b) : x \mapsto a \cdot x \cdot b
$
belongs to $\mathcal{N}^p_o( S_{p'} , S_1 )$ and satisfies $\nu^p_o(M(a,b)) \le \n{a}_{S_{2p}} \n{b}_{S_{2p}}$.
\end{proposition}

\begin{proof}
Write $a = (a_{ij})$ and  $b = (b_{ij})$.
Let us first observe that we may assume that only finitely many entries in each of these two infinite matrices are nonzero.
To that effect, suppose that we know the result for finitely supported matrices, and let $a, b \in S_{2p}$.
We can find sequences $(a^k), (b^k)$ in $S_{2p}$ of finitely supported matrices converging to $a$ and $b$, respectively, in the $S_{2p}$ norm.
Since for $x \in S_{p'}$ we have
$$
a^kxb^k - a^lxb^l = (a^k-a^l)xb^k + a^lx(b^k-b^l),
$$
it follows that $\big( M(a^k,b^k) \big)_k$ is a Cauchy sequence in $\mathcal{N}^p_o( S_{p'} , S_1)$ and therefore converges in the same space.
Since the norm in $\CB( S_{p'} , S_1 )$ is dominated by that of $\mathcal{N}^p_o( S_{p'} , S_1)$, it follows that the limit of $\big( M(a^k,b^k) \big)_k$ has to be $M(a,b)$, and therefore $M(a,b) \in \mathcal{N}^p_o( S_{p'} , S_1)$ with  $\nu^p_o(M(a,b)) \le \n{a}_{S_{2p}} \n{b}_{S_{2p}}$.

Therefore, from now on we assume $a=(a_{ij})_{i,j=1}^n$ and $b=(b_{ij})_{i,j=1}^n$.
Let $\varepsilon = [\varepsilon_{ij}]_{i,j=1}^n$ be the matrix of matrix units.
For $x = (x_{ij}) \in S_{p'}$,
\begin{align*}
(a  x  b)_{ij} &= \sum_{k,l} \varepsilon_{ij} \otimes a_{ik} x_{kl} b_{lj} \\
&= \sum_{k,l}  \varepsilon_{ij} \otimes a_{ik} \varepsilon_{kl}(x) b_{lj}.
\end{align*}
From the above calculation
$$
M(a,b)(x) = \widetilde{Q}^p\big( \varepsilon \otimes (a \cdot \varepsilon \cdot b) \big) (x),
$$
hence $M(a,b) = \widetilde{Q}^p\big( \varepsilon \otimes (a \cdot \varepsilon \cdot b) \big)$ and
$$
\nu^p_o\big( M(a,b) \big) \le \n{\varepsilon}_{S_{p'}^w[S_p]} \n{ a \cdot \varepsilon \cdot b }_{S_p[S_1]}.
$$
Note that
$$
{\varepsilon = [\varepsilon_{ij}]_{i,j=1}^n \in S_{p'}^w[S_p] } = \CB(S_{p},S_{p})
$$
is just the projection onto an initial block $S_p^n$, and thus $\n{\varepsilon}_{S_{p'}^w[S_p]} = 1$.

On the other hand, by \cite[Theorem 1.5 and Lemma 1.6]{Pisier-Asterisque-98},
$$
\n{a \cdot \varepsilon \cdot b}_{S^n_p[S_1]} \le \n{a}_{S^n_{2p}} \n{\varepsilon}_{M_n(S_1)} \n{b}_{S^n_{2p}}.
$$
Now,
$$
\varepsilon =  [\varepsilon_{ij}]_{i,j=1}^n \in M_n(S_1) = M_n(S_\infty') \subseteq \CB(S_\infty, M_n)
$$
is once again just a projection onto an initial block, so $\n{\varepsilon}_{M_n(S_1)} = 1$.
Therefore, we conclude that $\nu^p_o\big( M(a,b) \big) \le  \n{a}_{S_{2p}} \n{b}_{S_{2p}}$.
\end{proof}

The multiplication operators defined in the previous proposition are the canonical prototypes of completely right $p$-nuclear mappings: we show below that a mapping is completely right $p$-nuclear if and only if it admits a factorization through one such multiplication operator (similar to the Banach space framework as in \eqref{right p-nuclear}).

\begin{theorem} \label{right p nuclear factorization}
For a linear map $T : E \to F$ and $1\le p\le\infty$, the following are equivalent:
\begin{enumerate}
\item[(a)] $T$ is completely right $p$-nuclear.
\item[(b)] There exist $a,b \in S_{2p}$ such that $T$  admits a factorization
$$
\xymatrix{
E \ar[r]^T \ar[d]_{U} &F \\
S_{p'} \ar[r]_{M(a,b)} &S_1 \ar[u]_{V}
}
$$
\end{enumerate}
Moreover, in this case
$$
\nu^p_o(T) = \inf \big\{ \n{U}_{\cb} \n{V}_{\cb}  \n{a}_{S_{2p}} \n{b}_{S_{2p}}  \big\}
$$
where the infimum is taken over all factorizations as in (b).
\end{theorem}

\begin{proof}
(b) $\Rightarrow$ (a): This follows from Propositions \ref{proposition:ideal-property} and \ref{proposition:multiplication-operators-are-nuclear}.
\\
\\
(a) $\Rightarrow$ (b):  Let $1\le p<\infty$, assume that $T  \in \mathcal{N}^p_o(E;F)$ with $\nu^p_o(T) < 1$.
Then there exists $u \in E' \widehat{\otimes}_{d_p^o} F$ such that $J^p(u) = T$ and $\n{u}_{d_p^o} < 1$.
By Theorem \ref{thm:hard-Chevet-Saphar}, we can in turn find
$X \in   S_{p'}\widehat{\otimes}_{\min} E'$, $Y \in S_p[F]$
such that $u = Q^p( X\otimes Y )$, $\n{X}_{S_{p'}\widehat{\otimes}_{\min} E'}<1$ and $\n{Y}_{S_{p}[F]}<1$.
On the one hand, the representation of the minimal tensor product and its symmetry allows us to think of $X$ as a  mapping $U$ in $\mathcal{CB}(E, S_{p'})$  with $\n{U}_{\cb}<1 $.
On the other hand, by \cite[Thm. 1.5]{Pisier-Asterisque-98}, we can write $Y = a \cdot \bar{Y} \cdot b$ with $\n{a}_{S_{2p}}<1$, $\n{b}_{S_{2p}}<1$ and $\n{\bar{Y}}_{S_\infty[F]} < 1$.
Also, since $S_\infty[F] = S_\infty \widehat{\otimes}_{\min} F$ completely isometrically embeds into $\CB(S_1,F)$,
$\bar{Y}$ canonically induces a linear map $V : S_1 \to F$ with $\n{V}_{\cb} = \n{\bar{Y}}_{S_\infty[F]}$.
Chasing down the formulas it is easy to see that $T = V \circ  M(a,b) \circ U$ (in the spirit of \cite[Prop. 12.2.3]{Effros-Ruan-book}), giving us the desired factorization.
Finally, for $p=\infty$ this is derived directly from the definition of right $\infty$-nuclear with $M(a,b)=Id$.
\end{proof}

\section{operator $p$-compact mappings} \label{section completely compact mappings}

Following \cite{sinha2002compact,delgado2010operators} we say that a linear mapping between Banach spaces $T: \textbf X \to  \textbf Y$ is $p$-compact if $T(B_{\textbf X})$ is a relatively $p$-compact set. That is, there exists a sequence $(y_n)_n \in \textbf Y$  such that
\begin{equation}\label{T bola p-compacto}
T(B_{\textbf X})\subset \left\{\sum_{n=1}^{\infty} \alpha_n y_n \colon \sum_{n =1}^\infty \vert \alpha_n \vert^{p'} \leq 1\right\} \mbox{ and } \sum_{n=1}^{\infty} \Vert y_n \Vert^p < \infty.
\end{equation}

The class of $p$-compact mappings is denoted by $\mathcal K_p$, and endowed with the norm
\begin{equation}\label{norma kp Banach}
\kappa_p(T):=\inf\{\Vert (y_n)_n \Vert_{\ell_p(\textbf Y)}\},
\end{equation}
where the infimum runs all over the sequences $(y_n)_n \in \textbf Y$ as in Equation \eqref{T bola p-compacto}.

This notion should not be confused with the homonymous concept studied in the late seventies and early eighties by \cite{Pietsch-Operator-Ideals} and \cite{fourie1979banach} (see also \cite{Defant-Floret}).
Nowadays, this older class is sometimes referred to as ``classical $p$-compact'' mappings (as proposed in \cite{oja2012remark}).
We will not deal with such mappings in this paper.

The following characterization of $p$-compactness in terms of commutative diagrams is well-known to experts, but we were unable to find an explicit reference for it. Thus, we include a sketch of its proof for completeness.

\begin{proposition}
Let $\textbf X$ and $\textbf Y$ be Banach spaces, and $T: \textbf X \to \textbf Y$ a linear mapping. The following are equivalent:
\begin{enumerate}
\item $T: \textbf X \to \textbf Y$ is $p$-compact
\item There exist a Banach space $\textbf Z$, a mapping $\Theta \in \mathcal N^p(\textbf Z; \textbf Y)$ and a bounded mapping $R \in \mathcal{L}(\textbf X,\textbf Z/\ker\Theta)$ with $\Vert R \Vert \leq 1$ such that the following diagram commutes

\begin{equation}\label{p-compact factor}
\xymatrix{
\textbf X \ar[r]^T  \ar[rd]_R    & \textbf Y  & \ar[l]_{\Theta } \ar@{->>}[ld]^{\pi}  \textbf Z \\
 &   \textbf Z/\ker\Theta  \ar[u]^{\widetilde{\Theta}},  &
}
\end{equation}

\end{enumerate}
Moreover, $$\kappa_p(T) = \inf\{ \nu^p(\Theta): \Theta \in \mathcal N^p(\textbf Z; \textbf Y) \text{ as in } \eqref{p-compact factor}\}.$$

\begin{proof}
$(1) \Rightarrow (2)$: This is essentially contained in the proof of \cite[Proposition 2.9.]{galicer2011ideal} (which is based on \cite[Theorem 3.2.]{sinha2002compact}).
A careful look shows that it also holds $$\inf\{ \nu^p(\Theta): \Theta \in \mathcal N^p(\textbf Z; \textbf Y) \text{ as in } \eqref{p-compact factor}\} \leq \kappa_p(T).$$

$(2) \Rightarrow (1)$: Given $\varepsilon >0$ we have the following representation for $\Theta$:

$$ \Theta = \sum_{n\in \mathbb N} x_n' \otimes y_n,$$
where $(x_n)_{n \in \mathbb N} \in \ell_{p'}^w(\textbf Z')$, $(y_n)_{n \in \mathbb N} \in \ell_{p}(\textbf Y)$ such that $\Vert (x_n) \Vert_{\ell_{p'}^w(\textbf Z')} \cdot \Vert (y_n) \Vert_{\ell_{p}(\textbf Y)} \leq (1+\varepsilon) \nu^p(\Theta)$.

Let us show that $T= \tilde \Theta R$ is $p$-compact. Indeed, if $x \in B_{\textbf X}$ then $\Vert Rx \Vert_{\textbf Z/\ker\Theta} \leq 1$. Thus, there is $z \in \textbf Z$ such that $\pi z = Rx$, $\Vert z \Vert_{\textbf Z} \leq 1 + \varepsilon$. Then, $Tx=\tilde \Theta Rx= \tilde \Theta \pi z = \Theta z = \sum_{n \in \mathbb N} x_n'(z)y_n$. Since $$\Vert(x_n'(z))\Vert_{\ell_{p'}} \leq  \Vert (x_n) \Vert_{\ell_{p'}^w(\textbf Z')}  \Vert z \Vert_{\textbf Z} \leq (1 + \varepsilon) \Vert (x_n) \Vert_{\ell_{p'}^w(\textbf Z')},$$
this shows that $Tx$ lies in the $p$-convex hull of the sequence $(y_n \cdot (1+\varepsilon) \Vert (x_n) \Vert_{\ell_{p'}^w(\textbf Z')})_{n \in\mathbb N}$. This implies that $T$ is $p$-compact and moreover, $\kappa(T) \leq (1+ \varepsilon)^2 \nu^p(\Theta)$.
\end{proof}

\end{proposition}

We now look for a definition of $p$-compact mappings in the noncommutative framework. We choose to define it as in Equation \eqref{p-compact factor} since this will allow us to provide the ideal with an operator space structure.
Later on, we will see an equivalent description which is more obviously similar to the original Banach space definition (see Theorem \ref{thm:characterization-p-compact-mapping-via-p-compact-sets}).

\begin{definition}
Let $E$ and $F$ be operator spaces. A mapping $T \in \mathcal{CB}(E,F)$ is called \emph{operator $p$-compact} if there exist an operator space $G$, a completely right $p$-nuclear mapping $\Theta \in \mathcal{N}^p_o(G,F)$ and a completely bounded mapping $R \in \mathcal{CB}(E,G/\ker\Theta)$ with $\Vert R \Vert_{cb} \leq 1$ such that the following diagram commutes

\begin{equation}\label{p-compact o.s.}
\xymatrix{
E \ar[r]^T  \ar[rd]_R    & F  & \ar[l]_{\Theta } \ar@{->>}[ld]^{\pi}  G \\
 &   G/\ker\Theta  \ar[u]^{\widetilde{\Theta}},  &
}
\end{equation}

where $\pi$ stands for the natural 1-quotient mapping and $\widetilde{\Theta}$ is given by $\widetilde{\Theta}(\pi(g)) = \Theta(g)$.

\end{definition}

The set of all operator $p$-compact mappings from $E$ to $F$ is denoted by $\mathcal{K}_p^o(E,F)$. For $T \in  \mathcal{K}_p^o(E,F)$, we also define
$$\kappa_p^o(T):= \inf \{ \nu^p_o(\Theta) \},$$ where the infimum runs over all possible completely right $p$-nuclear mappings $\Theta \in \mathcal{N}^p_o(G,F)$  as in  \eqref{p-compact o.s.}.

\begin{proposition}
Let $E$ and $F$ be operator spaces. The set  $\mathcal{K}_p^o(E,F)$ is a linear subspace of $\mathcal{CB}(E,F)$ and $\kappa_p^o$ defines a norm for this space.

\end{proposition}

\begin{proof}
It is clear that  $\kappa_p^o(T)\ge 0$, for all $T\in \mathcal{K}_p^o(E,F)$. Also, $\kappa_p^o(T)= 0$ implies, for each $\varepsilon>0$, the existence of a commutative diagram as \eqref{p-compact o.s.}  with  $\nu^p_o(\Theta) <\varepsilon$. Then, $\|\Theta\|_{cb}<\varepsilon$ and so $\|\widetilde\Theta\|_{cb}<\varepsilon$ from which it is derived that $T=0$.

If $T\in \mathcal{K}_p^o(E,F)$ and $\lambda$ is a scalar, it is easy to see that $\lambda T\in \mathcal{K}_p^o(E,F)$ with $\kappa_p^o(\lambda T)=|\lambda| \kappa_p^o(T)$.

Now, let us consider  $T_1, T_2 \in \mathcal{K}^o_p(E,F)$. For each $i=1,2$ we have a commutative diagram:

\begin{equation*}
\xymatrix{
E \ar[r]^{T_i}  \ar[rd]_{R_i}    & F  & \ar[l]_{\Theta_i } \ar@{->>}[ld]^{\pi_i}  G_i \\
 &   G_i/\ker\Theta_i  \ar[u]^{\widetilde{\Theta_i}}  &
}
\end{equation*}
where $G_i$ is an operator space, $R_i\in \mathcal{CB}(E, G_i)$ with $\|R_i\|_{cb}\le 1$ and $\Theta_i \in \mathcal{N}^p_o(G_i,F)$. Define $G=G_1\oplus_\infty G_2$ and consider $\Theta:G\to F$ given by $\Theta (g_1,g_2)=\Theta_1 g_1+ \Theta_2 g_2$. If we denote, for each $i=1,2$, $\rho_i:G\to G_i$ the canonical restriction mapping we can write $\Theta=\Theta_1 \rho_1+ \Theta_2 \rho_2$. Since $\mathcal N^p_o$ is a mapping ideal, we derive that $\Theta\in \mathcal{N}^p_o(G,F)$ with $\nu^p_o(\Theta)\le \nu^p_o(\Theta_1)+\nu^p_o(\Theta_2)$.

 We define $\Lambda:G_1/\ker\Theta_1\oplus_\infty G_2/\ker\Theta_2\to G/\ker\Theta$ by $\Lambda(\pi_1 g_1, \pi_2 g_2)=\pi (g_1, g_2)$, where $\pi:G\to G/\ker\Theta$ is the natural 1-quotient mapping. It is clear that $\Lambda$ is well defined and that it is completely bounded with $\|\Lambda\|_{cb}\le 1$. And the same is true for the mapping $R:E\to G/\ker\Theta$, given by $Rx= \Lambda(R_1x,R_2x)$, for any $x\in E$.

Now, we can assemble the following diagram:

\begin{equation*}
\xymatrix{
E \ar[r]^{T_1+T_2}  \ar[rd]_R    & F  & \ar[l]_{\Theta } \ar@{->>}[ld]^{\pi}  G \\
 &   G/\ker\Theta  \ar[u]^{\widetilde{\Theta}}  &
}
\end{equation*}

Straightforward computations show that the diagram is commutative and hence $T_1+T_2$ belongs to $\mathcal{K}_p^o(E,F)$. Also, $\kappa_p^o(T_1+T_2)\le \nu^p_o(\Theta)\le \nu^p_o(\Theta_1)+\nu^p_o(\Theta_2)$, for every $\Theta_1$ and $\Theta_2$ ``admissible'' for the factorization of $T_1$ and $T_2$. Therefore, $\kappa_p^o(T_1+T_2)\le \kappa_p^o(T_1) + \kappa_p^o(T_2)$.

\end{proof}

From the definition, it is straightforward to show that any $T \in \mathcal{N}^p_o(E,F)$  belongs to $\mathcal{K}_p^o(E,F)$ with $\kappa_p^o(T) \leq \nu^p_o(T)$.

We are now interested in giving an operator space structure to $\mathcal{K}_p^o$ based on the relation, in the Banach space setting, of the right $p$-nuclear and $p$-compact ideals.

\begin{definition}
A mapping ideal $\mathcal{A}$ is \emph{surjective} if for each completely 1-quotient mapping $Q: G \twoheadrightarrow E$ we have  $T  \in \mathcal{A}(E,F)$ whenever $T  Q \in \mathcal{A}(G,F)$.
In this case, we have $\Vert T \Vert_{\mathcal{A}} = \Vert T  Q \Vert_{\mathcal{A}}$ for every $T \in M_n(\mathcal{A}(E,F))$.
\end{definition}

\begin{definition}
Given a mapping ideal $\mathcal{A}$, its \emph{surjective hull} $\mathcal{A}^{sur}$ is the smallest surjective mapping ideal that contains $\mathcal{A}$.
\end{definition}

We recall \cite[Proposition 2.12.2]{Pisier-Operator-Space-Theory} that for every operator space $E$ there is a set $I$ and a family $(n_i)_{i \in I} \subset \mathbb{N}$ such that $E$ is the quotient of $\ell_1(\{S_1^{n_i} : i \in I \})$. We denote the latter space by $Z_E$ and $Q_E : Z_E \twoheadrightarrow E$ the corresponding completely 1-quotient mapping. The space $Z_E$ is projective (see for example \cite[Chapter 24]{Pisier-Operator-Space-Theory}).

\begin{proposition}
For every $E$ and $F$ operator spaces we have the following characterization of the  \emph{surjective hull}
$$\mathcal{A}^{sur}(E,F)= \{ T \in \mathcal{CB}(E,F) : T  Q_E \in \mathcal{A}(E,F)\},$$
with $\Vert T \Vert_{\mathcal{A}^{sur}} = \Vert T  Q_E \Vert_{\mathcal{A}}$.
\end{proposition}

\begin{proof}
 We define $$\mathcal{U}(E,F) = \{ T \in \mathcal{CB}(E,F) : T  Q_E \in \mathcal{A}(E,F)\},$$
 endowed with the operator space norm given by $\Vert T \Vert_{\mathcal{U}} = \Vert T  Q_E \Vert_{\mathcal{A}}$, for every $T \in M_n(\mathcal{U}(E,F))$.

 We now show that $\mathcal{U}$ is a mapping ideal. Indeed, if $T \in M_n(\mathcal{U}(E,F))$ we obviously have $$\Vert T \Vert_{\mathcal{U}(E,F)} = \Vert T  Q_E \Vert_{\mathcal{A}} \geq \Vert T  Q_E \Vert_{cb} = \Vert T  \Vert_{cb}.$$
Let $T \in M_n(\mathcal{U}(E,F))$, $R \in \mathcal{CB}(E_0,F)$, $S \in \mathcal{CB}(E,F_0)$.
Since $Z_{E_0}$ is projective, given $\varepsilon >0$, there is a lifting $L_{\varepsilon} \in \mathcal{CB}(Z_{E_0},Z_E)$ of $R  Q_{E_0}$ with $\Vert L_{\varepsilon} \Vert_{cb} \leq (1+\varepsilon) \Vert R \Vert_{cb}$ such that  the following diagram commutes

$$
\xymatrix{
E_0   \ar[r]^R    & E   \ar[r]^{T_{i,j} }   & F  \ar[r]^S & F_0 \\
Z_0  \ar@{->>}[u]^{Q_{E_0}}  \ar@{.>}[r]_{L_{\varepsilon}}  & Z_E   \ar@{->>}[u]_{Q_{E}} &  &
}.
$$

Then, for every $1 \leq i, j \leq n$ we have $ST_{i,j}RQ_{E_0} = ST_{i,j}Q_{E}L_{\varepsilon} \in \mathcal{A}(Z_{E_0},F_0)$ and hence  $ ST_{i,j}R \in \mathcal{U}(E_0,F_0)$.
Moreover, by the ideal property of $\mathcal{A}$,

\begin{align*}
 \Vert S_n T R \Vert_{M_n(\mathcal{U}(E_0,F_0))} & = \Vert S_n T R Q_{E_0} \Vert_{M_n(\mathcal{A}(Z_{E_0},F_0))} =  \Vert S_n T  Q_{E} L_{\varepsilon} \Vert_{M_n(\mathcal{A}(Z_{E_0},F_0))} \\
 & \leq  \Vert S_n T  Q_{E}  \Vert_{M_n(\mathcal{A}(Z_{E},F_0))} (1+\varepsilon) \Vert R \Vert_{cb}  \\
 & \leq (1+\varepsilon) \Vert S \Vert_{cb} \Vert T \Vert_{M_n(\mathcal{U}(E,F))} \Vert R \Vert_{cb}.
\end{align*}
We now prove that the mapping ideal $\mathcal{U}$ is surjective. Let $Q : G \twoheadrightarrow E$ be a complete 1-quotient mapping and $T \in \mathcal{CB}(E,F)$ such that $TQ \in \mathcal{U}(G,F)$.
Since $Z_E$ is projective and $QQ_G : Z_G \twoheadrightarrow E$ is a complete 1-quotient mapping, given $\varepsilon >0$, there is a lifting  of $Q_E$, $L_{\varepsilon}: Z_E \to Z_G$ with cb-norm less than or equal to $1+ \varepsilon$.

We have $TQ_E = T Q Q_G L_{\varepsilon} \in \mathcal{A}(Z_E,F)$ since $TQQ_G \in \mathcal{A}(Z_G,F)$. Then, $T \in \mathcal{U}(E,F)$.
Also, if $T \in M_n(\mathcal{U}(E,F))$ we obtain

\begin{align*}
 \Vert T \Vert_{\mathcal{U}} & =  \Vert TQ_E  \Vert_{\mathcal{A}} =  \Vert TQ Q_G L_{\varepsilon}  \Vert_{\mathcal{A}} \leq (1 + \varepsilon) \Vert TQ Q_G \Vert_{\mathcal{A}} \\
 &  = (1 + \varepsilon) \Vert TQ \Vert_{\mathcal{U}} \leq (1+ \varepsilon) \Vert T \Vert_{\mathcal{U}}.
\end{align*}

Note that, by the definition, $\mathcal{U}(E,F) \subset \mathcal{A}^{sur}(E,F)$: if $T \in \mathcal{U}(E,F)$, we have $TQ_E \in \mathcal{A}(Z_E,F) \subset \mathcal{A}^{sur}(Z_E,F)$. Since $\mathcal{A}^{sur}$ is surjective we obtain that $T \in \mathcal{A}^{sur}(E,F)$.

We have shown that $\mathcal{U}$ is a surjective mapping ideal (which obviously contains $\mathcal A$). Then, by minimality we have the reverse inclusion.
\end{proof}

\begin{proposition}\label{projectivo coinciden}
Let $E$ be a projective operator space. Then, $T \in  \mathcal{K}_p^o(E,F)$ if and only if ${T \in \mathcal{N}^p_o(E,F)}$ and $\kappa_p^o(T) = \nu^p_o(T)$.
\end{proposition}

\begin{proof}
 Let $T  \in  \mathcal{K}_p^o(E,F)$, and consider a factorization as in \eqref{p-compact o.s.}. Given $\varepsilon>0$, since $E$ is projective, there is a lifting of $R$, $\widetilde R_{\varepsilon} \in \mathcal{CB}(E,G)$ with $\Vert \widetilde R_{\varepsilon} \Vert \leq 1 + \varepsilon$, such that the following diagram commutes

 \begin{equation*}
\xymatrix{
E \ar[r]^T \ar@{.>}@(ur,ul)[rr]^{\widetilde R_{\varepsilon}} \ar[rd]_R    & F  & \ar[l]_{\Theta } \ar@{->>}[ld]^{\pi}  G \\
 &   G/\ker\Theta  \ar[u]^{\widetilde{\Theta}}.  &
}
\end{equation*}

 Then, $T = \Theta \widetilde R_{\varepsilon} \in \mathcal{N}^p_o(E,F)$ and $\nu^p_o(T) \leq (1 + \varepsilon) \nu^p_o(\Theta)$. Observe that this holds for every $\Theta$ verifying Equation \eqref{p-compact o.s.}. Thus, by definition of the norm $\kappa_p^o(T)$, we obtain $\nu^p_o(T) \leq  (1 + \varepsilon) \kappa_p^o(T)$ and the result follows.
\end{proof}

\begin{proposition}
Let $E$ and $F$ be operator spaces. Then $T \in  \mathcal{K}_p^o(E,F)$ if and only if $T  Q_E \in \mathcal{N}^p_o(Z_E,F)$ and $\kappa_p^o(T) = \nu^p_o(T Q_E)$.
\end{proposition}

\begin{proof}
If $T \in \mathcal{K}_p^o(E,F)$, then $TQ_E \in \mathcal{K}_p^o(Z_E,F)$. Now by Proposition \ref{projectivo coinciden}, $TQ_E \in \mathcal{N}^p_o(Z_E,F)$ and also $\nu^p_o(TQ_E) = \kappa_p^o(TQ_E) \leq \kappa_p^o(T)$, since $Z_E$ is projective.

Reciprocally, if $TQ_E \in \mathcal{N}^p_o(Z_E,F)$, we define $R \in \mathcal{CB}(E,Z_E/\ker\Theta)$ in the following way: $Rx := \pi y$ where $Q_Ey=x$. It is not difficult to check that $R$ is well defined and also  $\Vert R \Vert_{cb} \leq 1$. If we denote $\Theta :=TQ_E$ we have
$\widetilde \Theta R x = \widetilde \pi y = \Theta y = TQ_Ey = Tx$ for every $x \in E$.
Thus, the following diagram commutes

 \begin{equation*}
\xymatrix{
E \ar[r]^T \ar@{.>}[rd]^{R}  & F  & \ar[l]_{\Theta } \ar@{->>}[ld]^{\pi}  Z_E \\
 &   Z_E/\ker\Theta  \ar[u]^{\widetilde{\Theta}}.  &
}
\end{equation*}

Therefore, $T$ is an operator $p$-compact mapping and $\kappa_p^o(T) \leq \nu^p_o(TQ_E)$. This concludes the proof.
\end{proof}

We have shown that for every $E$ and $F$ operator spaces, we have the equality (as Banach spaces)
$$\mathcal{K}_p^o(E,F) = (\mathcal{N}^p_o)^{sur}(E,F).$$

This induces a natural operator space structure for $\mathcal{K}_p^o(E,F)$. Indeed, if $T \in M_n(\mathcal{K}_p^o(E,F))$, we define
$$\Vert T \Vert_{M_n(\mathcal{K}_p^o(E,F))} := \Vert T \Vert_{M_n((\mathcal{N}^p_o)^{sur}(E,F))} = \Vert T  Q_E \Vert_{M_n(\mathcal{N}^p_o(Z_E,F))}.$$

As a consequence we can say that $\mathcal{K}_p^o$ is a mapping ideal through the following identification:

\begin{equation}\label{Npsur}
\mathcal{K}_p^o=(\mathcal{N}^p_o)^{sur}.
\end{equation}

In the Banach space setting, a continuous linear mapping is compact if it sends the unit ball into a relatively compact set. Equivalently, it sends the unit ball into the closure of the convex hull of a null sequence. In the operator space framework, if we look at how a mapping acts on the ``matrix unit ball'' (see definition below) the previous conditions are not equivalent. Webster \cite{webster1997local, Webster1998} noticed this and studied several notions of compactness: operator compactness, strong operator compactness, and matrix compactness. Let us recall the definition of ``operator compactness'' that later we will  naturally extend  to the case of $p$-compactness.

If $E$ is an operator space, a \textit{matrix set} is a sequence of sets $\mathbf{K}=(K_n)$, where $K_n\subset M_n(E)$, for all $n$. The closure of a matrix set means the closure of each set of the sequence in the corresponding matrix space. The \textit{matrix unit ball} is the matrix set $\left(B_{M_n(E)}\right)$.

Given $X \in \mathcal S_\infty[E]$, Webster defined the \textit{absolutely matrix convex hull} of $X$ to be the matrix set $\mathbf{co}(X)$ where
$$\big(\mathbf{co}(X)\big)_k = \{v \in M_k(E) : \exists\sigma \in M_k(M_{\infty}^{fin}), \; \Vert \sigma \Vert_{M_k(S_{1})} \leq 1 \textrm{ such that } (\sigma \otimes id)X=v   \}.$$

A mapping $T\in\mathcal{CB}(E,F)$ is \textit{operator compact} if the image by $T$ of the matrix unit ball of $E$ is contained in $\overline{\mathbf{co}(Y)}$, for some $Y\in \mathcal  S_\infty[F]$.

Based on the work of Webster we now define the notion of  $p$-absolutely matrix convex hull.

\begin{definition}
Let  $E$ be an operator space and $X\in S_p [E]$, for $1 \leq p \leq \infty$, we define the $p$-absolutely matrix convex hull of $X$  to be the matrix set $\mathbf{co}_p(X)$ where
$$\big(\mathbf{co}_p(X)\big)_k = \{v \in M_k(E) : \exists\sigma \in M_k(M_{\infty}^{fin}), \; \Vert \sigma \Vert_{M_k(S_{p'})} \leq 1 \textrm{ such that } (\sigma \otimes id)X=v   \}.$$

\end{definition}

We say that a matrix set $\mathbf{K}=(K_n)$ in $E$ is \emph{operator $p$-compact} if $\mathbf{K}$ is is contained in $\overline{\mathbf{co}_p(X)}$, for some $X\in S_p [E]$.

We are using the term ``operator $p$-compact'' for two seemingly different notions, but they will turn out to be the same.
If we denote by $\mathbf{B}_E$ the matrix unit ball of $E$,
we will prove that  $T: E \to F$ is \textit{operator $p$-compact} if and only if  the image under $T$ of the matrix unit ball of $E$ is an operator $p$-compact matrix set in $F$.

Note that the $\infty$-absolutely matrix convex hull is Webster's absolutely matrix convex hull, and thus operator $\infty$-compact mappings are the same as operator compact mappings.

\begin{remark} \label{remark capsula Webster}
Given $1\le p\le\infty$, for any $X \in S_p[E]$ and $k \in \N$ we have
$$
\left(\overline{\mathbf{co}_p(X)}\right)_k = \{v \in M_k(E) : \exists\sigma \in M_k(S_{p'}), \; \Vert \sigma \Vert_{M_k(S_{p'})} \leq 1 \textrm{ such that } (\sigma \otimes id)X=v   \},$$
where, in the case $p=1$ the space $S_{p'}$ should be replaced by $M_\infty$.
\end{remark}

\begin{proof}
Indeed, let us denote by $C_k$ the set on the right hand side. For $1<p\le\infty$, it is enough to show that $C_k$ is closed. Consider
$\Psi: M_k(S_{p'}) \to M_k(E)$ defined by $\sigma \mapsto (\sigma \otimes id)(X)$.
Note that the image by $\Psi$ of $B_{M_k(S_{p'})}$, the closed unit ball of $M_k(S_{p'})$, is exactly $C_k$.
Take a sequence $(\sigma_l)_l \in B_{M_k(S_{p'})}$ such that $\Psi(\sigma_l) \to v \in M_k(E)$. We have to see that $v= \Psi(\sigma)$ for certain $\sigma \in M_k(S_p')$ with $\Vert \sigma \Vert_{M_k(S_{p'})} \leq 1$.
Since $B_{M_k(S_{p'})}$ is weak$^{*}$ sequentially compact (since the predual of $M_k(S_{p'})$ is separable) there is a subsequence $(\sigma_{l_j})_j$ weak$^{*}$ convergent to an element $\sigma \in B_{M_k(S_{p'})}$.
By uniqueness of the limit our conclusion follows once proving that for all matrices $v' \in M_k(E')$ we have the convergence of the scalar pairing (in the sense of  \cite[1.1.24]{Effros-Ruan-book})
$$
\langle  \Psi(\sigma_{l_j}),v' \rangle  \to \langle  \Psi(\sigma),v' \rangle.
$$
Now,
$$ \langle  \Psi(\sigma_{l_j}),v' \rangle = \langle ( \sigma_{l_j} \otimes id)(X),v' \rangle = \langle  \sigma_{l_j}, (id \otimes v')(X) \rangle \to \langle  \sigma, (id \otimes v')(X) \rangle  = \langle \Psi(\sigma),v' \rangle.
$$
For $p=1$, the previous argument (replacing $S_{p'}$ by $M_\infty$) shows that $\left(\overline{\mathbf{co}_1(X)}\right)_k\subset C_k$. The other inclusion runs as follows. Recall, using the notation as in Theorem \ref{Thm:Q-tilde}, that any $X \in S_1[E]$ can be approximated by its ``truncations'' $\{\widetilde{\tau}^m(X)\}_m$ \cite[Lemma 1.12]{Pisier-Asterisque-98} and denote by $\rho^m:M_\infty\to M_m$ the truncated mapping between these matrix spaces. Now, for $v=(\sigma \otimes id)X \in C_k$ with $\sigma\in M_k(M_\infty)$, we have to see that $\{((\rho^m)_k\sigma \otimes id)X\}_m\subset \mathbf{co}_1(X)$ approximates $v$. Indeed,
\begin{align*}
\|v-((\rho^m)_k\sigma \otimes id)X\|_{M_k(E)}  & =  \|((\sigma-(\rho^m)_k\sigma)\otimes id)X\|_{M_k(E)} =  \|(\sigma\otimes id) (X-(\widetilde{\tau}^m)_k(X))\|_{M_k(E)} \\
&  \le \|X-(\widetilde{\tau}^m)_k(X)\|_{M_k(S_1[E])} \underset{m\to\infty}{\longrightarrow} 0.
\end{align*}
\end{proof}

\bigskip

The following result shows that the first definition of $p$-compactness (the one which involves the factorization through completely right $p$-nuclear mappings)
and the definition based on Webster's work are the same.
We highlight the analogy presented between Equations \eqref{norma kp Banach} and \eqref{kpo}, below.

\begin{theorem} \label{thm:characterization-p-compact-mapping-via-p-compact-sets}
Let $T: E \to F$ be a completely bounded mapping and $1\le p\le \infty$. Then $T$ is operator $p$-compact if and only if $T(\mathbf{B}_E)$ is an operator $p$-compact matrix set in $F$. Moreover,

\begin{equation}\label{kpo}
\kappa_p^o(T) = \inf\{ \Vert Y \Vert_{S_p[F]} : T(\mathbf{B}_E) \subset \overline{\mathbf{co}_p(Y)}\}.
\end{equation}
\end{theorem}

\begin{proof}
Suppose $T$ is operator $p$-compact. Without loss of generality we suppose that $\kappa_p^o(T) < 1$.
By the commutative diagram given in \eqref{p-compact o.s.} there is a completely right $p$-nuclear mapping $\Theta:G\to F$  with $\nu^p_o$-norm less than $1$.
It suffices to see that $\Theta$ is operator $p$-compact, since $T(\mathbf{B}_E)$ is contained in $\Theta(\mathbf{B}_G)$.
By Theorem \ref{thm:hard-Chevet-Saphar} we can write $\Theta = J^p\circ Q^p (X\otimes Y)$ where $\Vert X \Vert_{S_{p'}\widehat\otimes_{\min} G'} < 1$  and $\Vert Y \Vert_{S_p[F]} < 1$.
For $g \in M_n(G)$ of norm less than one, we will show that $\Theta_n g \in (\overline{\mathbf{co}_p(Y)})_n$.
Indeed,  for $1<p\le \infty$, since $X  \in S_{p'}\widehat\otimes_{\min} G'$ we derive $X'\circ \iota_G\in \mathcal{CB}(G, S_{p'})$. For $p=1$, being $X  \in S_{\infty}\widehat\otimes_{\min} G'$ we obtain $X'\circ \iota_G\in \mathcal{CB}(G, M_{\infty})$.
Consider $\sigma := \left(X'\circ \iota_G\right)_n g \in M_n(S_{p'})$ (replacing $S_{p'}$ by $M_\infty$ in the case $p=1$), then, $\Vert \sigma \Vert_{M_n(S_{p'})} < 1$   and it is not hard to check that
$$\Theta_n g =  (\sigma \otimes id) (Y).
$$
Then, $T(\mathbf{B}_E) \subset \Theta(\mathbf{B}_G) \subset \overline{\mathbf{co}_p(Y)}$ and  hence $ \inf\{ \Vert Y \Vert_{S_p[F]} : T(\mathbf{B}_E) \subset \overline{\mathbf{co}_p(Y)}\} < 1$.

For the converse, let $T: E \to F$ be an operator $p$-compact mapping with $ \inf\{ \Vert Y \Vert_{S_p[F]} : T(\mathbf{B}_E) \subset \overline{\mathbf{co}_p(Y)}\}< 1$. We begin by considering the case $1<p<\infty$. Denote by $\Theta : S_{p'} \to F$, the operator given by $J^p \circ B^p (id,  Y)$, which by definition is in $\mathcal N^p_o$ since $id \in S_{p'}^w[S_p]=\mathcal{CB}(S_p, S_p)$. Observe that by Corollary \ref{defB_p} (1), $\nu_o^p(\Theta) < 1$.  Remark \ref{remark capsula Webster} asserts that for each $x \in E$ there is $\sigma \in S_{p'}$ such that $Tx = (\sigma \otimes id) (Y)$. Defining  $Rx:=[\sigma]$  we have the following commutative diagram

 \begin{equation*}
\xymatrix{
E \ar[r]^T\ar[rd]_R    & F  & \ar[l]_{\Theta } \ar@{->>}[ld]^{\pi}  S_{p'} \\
 &   S_{p'}/\ker\Theta  \ar[u]^{\widetilde{\Theta}},  &
}.
\end{equation*}
Note that $R$ is well defined. Moreover, if $x \in B_{M_k(E)}$ by hypothesis
$T_k x =( \widetilde{ \sigma} \otimes id)(Y)$, for certain $\widetilde \sigma $ with $\Vert \widetilde \sigma  \Vert_{M_k(S_{p'})} \leq 1$.
Checking coordinates we have $T(x_{i_j})= ( \widetilde{ \sigma}_{i,j} \otimes id)(Y)$ and therefore $R(x_{i,j}) = [ \widetilde{ \sigma}_{i,j} ]$. Hence, $R_k x = ([\widetilde{ \sigma}_{i,j}])_{i,j})$ and then $\Vert R \Vert_{cb} \leq 1$.
We conclude that $T \in \mathcal K^o_p(E;F)$ and $\kappa_p^o(T) \leq \nu_o^p(\Theta) < 1$.

For $p=1$, we have $Y \in S_1[F]=S_1\widehat\otimes_{proj} F$. Since the projective tensor norm satisfies the Embedding Lemma (see \cite{dimant2015biduals} or \cite{CDDG}), there is a completely isometric embedding $\kappa: S_1\widehat\otimes_{proj} F\to S_1^{''}\widehat\otimes_{proj} F$. Now,  $\Theta= J^1\circ \kappa(Y)$ belongs to $\mathcal N^1_o(M_\infty,F)$, where  $J^1: S_1^{''}\widehat\otimes_{proj} F\to S_1^{''}\widehat\otimes_{\min} F$ is the canonical mapping. The rest of the argument follows as in the previous case.

For $p=\infty$, we have $Y \in S_\infty [F]$ and so $B^\infty (id,  Y)$ belongs to $S_\infty \widehat\otimes_{d_\infty^o} F$. Since $d_\infty^o$ is finitely generated and $S_\infty$ is locally reflexive, we can appeal again to the Embedding Lemma \cite{CDDG} to get a completely isometric embedding $\kappa: S_\infty\widehat\otimes_{d_\infty^o} F\to M_\infty\widehat\otimes_{d_\infty^o} F$. Hence, $\Theta=J^\infty \circ\kappa\circ B^\infty (id,  Y)$ belongs to $\mathcal N_o^\infty(S_1,F)$ and the proof finishes as in the two previous cases.
\end{proof}

The previous theorem, together with the relation presented in Equation \eqref{Npsur}, allows us to endow Webster's class
(of operator compact mappings) with a natural operator space structure, showing  that this  class is indeed a mapping ideal. Namely, it is exactly the ideal $\mathcal{K}_{\infty}^o$.

\section*{Some final questions} \label{section questions}

We present some open questions regarding the mapping ideal $\mathcal K_p^o$. \smallskip

In parallel to the Banach space case, we say that a mapping ideal $(\mathfrak{A},\mathbf{A})$ and  an operator space tensor norm $\alpha$  are \emph{associated}, denoted $(\mathfrak{A},\mathbf{A}) \sim \alpha$ if for every pair of finite-dimensional operator spaces $(E,F)$ we have a complete isometry
$
\mathfrak{A}(E,F) = E' \otimes_{\alpha} F,
$
given by the canonical map.
Notice that since this definition is based only on finite-dimensional spaces, two different mapping ideals can be associated to the same tensor norm (and vice versa).
For example, from Definition \ref{p-nuclear def} it is clear that $d_p^o$ is associated to the mapping ideal $\mathcal N_o^p$.

Recall that an operator space tensor norm $\alpha$ is called left-injective if for any operator spaces $E_1, E_2, F$ and a complete isometry $i:E_1 \to E_2$, the mapping
$$
   i\otimes id_F : E_1 \otimes_{\alpha} F \to E_2 \otimes_{\alpha} F
 $$
is  completely isometric.
Given an operator space tensor norm we denote by $/ \alpha$ the greatest left-injective tensor norm that is dominated by $\alpha$. This essentially mimics the notion given in the context of normed spaces deeply described in \cite[Chapter 20]{Defant-Floret} and \cite[Section 7.2]{Ryan}.

In the Banach space setting, it is shown in  \cite[Theorem 3.3]{galicer2011ideal} that $/ d_p$ is the tensor norm associated to $\mathcal K_p$. Therefore, is natural to ask if an analogous result holds in the operator space framework.

  \begin{question} \label{preg1}
   Is the operator space tensor norm $/ d_p^o$ associated to  the mapping ideal $\mathcal{K}_p^o$?
  \end{question}

  This question is a particular case of the following more general one (whose statement is valid in the Banach space context):
  \begin{question}\label{preg2}
   Let $\mathcal A$ be a mapping ideal and $\alpha$ its associated operator space tensor norm. Is $/ \alpha$ associated to $\mathcal A^{sur}$?
  \end{question}

Note that     $d_p^o \sim \mathcal N_o^p$ and, by Equation \eqref{Npsur},  we have  $\mathcal K_p^o = (\mathcal N_o^p)^{sur}$. Thus an affirmative answer to this question would provide also an answer to the first one.
Result of this kind for the Banach space setting, although sometimes hard to see it at first glance, are based on \emph{local techniques} (e.g., \cite[Proposition 20.9]{Defant-Floret}) which are no longer valid in the operator space framework.

In \cite{CDDG} it is shown that if $\beta \sim \mathcal A^{sur}$ then $\beta$ is left-injective. Therefore, since $\mathcal A \subset \mathcal A^{sur}$ we have that $\beta \leq / \alpha$. An affirmative answer to Question \ref{preg2} is obtained in  \cite{CDDG} for the particular case where $\alpha$ is left accessible (definition similar to the Banach space case).
It is known in the Banach space setting that $d_p$ satisfies this property (\cite[Theorem 21.5]{Defant-Floret}, \cite[Proposition 7.21.]{Ryan}), but unfortunately we do not know yet whether this also holds for the operator space counterpart $d_p^o$.

Another interesting question, which involves the mapping ideal $\mathcal K_{\infty}^o$, was posed by Webster in his thesis \cite[Section 4.1]{webster1997local}:

\begin{problem}{Webster:}\label{problema webster}
``It is an open question as to whether the space of operator compact maps must be closed in the cb-norm topology''.
\end{problem}
In other words, if a sequence of operator compact mappings $(T_n)_{n \in \N}$ converges to $T$ in the completely bounded norm ($\Vert T - T_n \Vert_{cb} \to 0$), does it imply that $T$ is also an operator compact map?

By Theorem \ref{thm:characterization-p-compact-mapping-via-p-compact-sets}, as a direct consequence of the Open Mapping Theorem we see that Problem \ref{problema webster} can be reformulated in terms of the following question.

\begin{question}
 Are the norms $\kappa_{\infty}^o$ and $\Vert \cdot \Vert_{cb}$ equivalent on $\mathcal{K}_{\infty}^o$?
\end{question}


\begin{thebibliography}{10}

\bibitem{ain2012compact}
K.~Ain, R.~Lillemets, and E.~Oja.
\newblock Compact operators which are defined by $\ell_p$-spaces.
\newblock {\em Quaestiones Mathematicae}, 35(2):145--159, 2012.

\bibitem{ain2015p}
K.~Ain and E.~Oja.
\newblock On $(p, r)$-null sequences and their relatives.
\newblock {\em Mathematische Nachrichten}, 288(14-15):1569--1580, 2015.

\bibitem{aron2016behavior}
R.~Aron, E.~{\c{C}}al{\i}{\c{s}}kan, D.~Garc{\'\i}a, and M.~Maestre.
\newblock Behavior of holomorphic mappings on $p$-compact sets in a {B}anach
  space.
\newblock {\em Transactions of the American Mathematical Society},
  368(7):4855--4871, 2016.

\bibitem{aron2010p}
R.~Aron, M.~Maestre, and P.~Rueda.
\newblock $p$-compact holomorphic mappings.
\newblock {\em Revista de la Real Academia de Ciencias Exactas, Fisicas y
  Naturales. Serie A. Matematicas}, 104(2):353--364, 2010.

\bibitem{aron2011p}
R.~Aron and P.~Rueda.
\newblock $p$-compact homogeneous polynomials from an ideal point of view.
\newblock {\em Contemp. Math}, 547:61--71, 2011.

\bibitem{aron2012ideals}
R.~Aron and P.~Rueda.
\newblock Ideals of homogeneous polynomials.
\newblock {\em Publications of the Research Institute for Mathematical
  Sciences}, 48(4):957--970, 2012.

\bibitem{Blecher-1992}
D.~P. Blecher.
\newblock Tensor products of operator spaces. {II}.
\newblock {\em Canad. J. Math.}, 44(1):75--90, 1992.

\bibitem{Blecher-LeMerdy-book}
D.~P. Blecher and C.~Le~Merdy.
\newblock {\em Operator algebras and their modules---an operator space
  approach}, volume~30 of {\em London Mathematical Society Monographs. New
  Series}.
\newblock The Clarendon Press Oxford University Press, Oxford, 2004.
\newblock Oxford Science Publications.

\bibitem{Blecher-Paulsen-Tensors}
D.~P. Blecher and V.~I. Paulsen.
\newblock Tensor products of operator spaces.
\newblock {\em J. Funct. Anal.}, 99(2):262--292, 1991.

\bibitem{CD-Chevet-Saphar-OS}
J.~A. Ch{\'a}vez-Dom{\'{\i}}nguez.
\newblock The {C}hevet-{S}aphar tensor norms for operator spaces.
\newblock {\em Houston J. Math.}, 42(2):577--596, 2016.

\bibitem{CDDG}
J.~A. Ch{\'a}vez-Dom{\'{\i}}nguez, V.~Dimant, and D.~Galicer.
\newblock Operator space tensor norms.
\newblock {\em Article in preparation}.

\bibitem{choi2010dual}
Y.~S. Choi and J.~M. Kim.
\newblock The dual space of {$( L(X, Y)$, $\tau_p$)} and the $p$-approximation
  property.
\newblock {\em Journal of Functional Analysis}, 259(9):2437--2454, 2010.

\bibitem{Defant-Floret}
A.~Defant and K.~Floret.
\newblock {\em Tensor norms and operator ideals}, volume 176 of {\em
  North-Holland Mathematics Studies}.
\newblock North-Holland Publishing Co., Amsterdam, 1993.

\bibitem{delgado2009p}
J.~M. Delgado, E.~Oja, C.~Pineiro, and E.~Serrano.
\newblock The p-approximation property in terms of density of finite rank
  operators.
\newblock {\em Journal of Mathematical Analysis and Applications},
  354(1):159--164, 2009.

\bibitem{delgado2010density}
J.~M. Delgado, C.~Pineiro, and E.~Serrano.
\newblock Density of finite rank operators in the {B}anach space of $p$-compact
  operators.
\newblock {\em Journal of Mathematical Analysis and Applications},
  370(2):498--505, 2010.

\bibitem{delgado2010operators}
J.~M. Delgado, C.~Pineiro, and E.~Serrano.
\newblock Operators whose adjoints are quasi $ p $-nuclear.
\newblock {\em Studia Mathematica}, 197:291--304, 2010.

\bibitem{Diestel-Jarchow-Tonge}
J.~Diestel, H.~Jarchow, and A.~Tonge.
\newblock {\em Absolutely summing operators}, volume~43 of {\em Cambridge
  Studies in Advanced Mathematics}.
\newblock Cambridge University Press, Cambridge, 1995.

\bibitem{dimant2015biduals}
V.~Dimant and M.~Fern{\'a}ndez-Unzueta.
\newblock Biduals of tensor products in operator spaces.
\newblock {\em Studia Mathematica}, 230(2):165--185, 2015.

\bibitem{Effros-Junge-Ruan}
E.~G. Effros, M.~Junge, and Z.-J. Ruan.
\newblock Integral mappings and the principle of local reflexivity for
  noncommutative {$L^1$}-spaces.
\newblock {\em Ann. of Math. (2)}, 151(1):59--92, 2000.

\bibitem{Effros-Ruan-book}
E.~G. Effros and Z.-J. Ruan.
\newblock {\em Operator spaces}, volume~23 of {\em London Mathematical Society
  Monographs. New Series}.
\newblock The Clarendon Press Oxford University Press, New York, 2000.

\bibitem{fourie2018injective}
J.~H. Fourie.
\newblock Injective and surjective hulls of classical $p$-compact operators
  with application to unconditionally $p$-compact operators.
\newblock {\em Studia Mathematica}, 240:147--159, 2018.

\bibitem{fourie1979banach}
J.~H. Fourie and J.~Swart.
\newblock {B}anach ideals of $p$-compact operators.
\newblock {\em manuscripta mathematica}, 26(4):349--362, 1979.

\bibitem{fourie2017weak}
J.~H. Fourie and E.~D. Zeekoei.
\newblock On weak-star p-convergent operators.
\newblock {\em Quaestiones Mathematicae}, pages 1--17, 2017.

\bibitem{galicer2011ideal}
D.~Galicer, S.~Lassalle, and P.~Turco.
\newblock The ideal of $ p $-compact operators: a tensor product approach.
\newblock {\em Studia Mathematica}, 211:269--286, 2012.

\bibitem{grothendieck1955produits}
A.~Grothendieck.
\newblock {\em Produits tensoriels topologiques et espaces nucl{\'e}aires},
  volume~16.
\newblock American Mathematical Soc., 1955.

\bibitem{grothendieck1956resume}
A.~Grothendieck.
\newblock {\em R{\'e}sum{\'e} de la th{\'e}orie m{\'e}trique des produits
  tensoriels topologiques}.
\newblock Soc. de Matem{\'a}tica de S{\~a}o Paulo, 1956.

\bibitem{Junge-Habilitationschrift}
M.~Junge.
\newblock Factorization theory for spaces of operators.
\newblock Habilitation Thesis. Kiel, 1996.

\bibitem{junge1999factorization}
M.~Junge.
\newblock {\em Factorization theory for spaces of operators}.
\newblock Institut for Matematik og Datalogi Odense Universitet, 1999.

\bibitem{Junge-Parcet-Maurey-factorization}
M.~Junge and J.~Parcet.
\newblock Maurey's factorization theory for operator spaces.
\newblock {\em Math. Ann.}, 347(2):299--338, 2010.

\bibitem{kim2014unconditionally}
J.~M. Kim.
\newblock Unconditionally $ p $-null sequences and unconditionally $ p
  $-compact operators.
\newblock {\em Studia Mathematica}, 224:133--142, 2014.

\bibitem{lassalle2016weaker}
S.~Lassalle, E.~Oja, and P.~Turco.
\newblock Weaker relatives of the bounded approximation property for a {B}anach
  operator ideal.
\newblock {\em Journal of Approximation Theory}, 205:25--42, 2016.

\bibitem{lassalle2012p}
S.~Lassalle and P.~Turco.
\newblock On $p$-compact mappings and the p-approximation property.
\newblock {\em Journal of Mathematical Analysis and Applications},
  389(2):1204--1221, 2012.

\bibitem{lassalle2013banach}
S.~Lassalle and P.~Turco.
\newblock The {B}anach ideal of {$A$}-compact operators and related
  approximation properties.
\newblock {\em Journal of Functional Analysis}, 265(10):2452--2464, 2013.

\bibitem{lassalle2017null}
S.~Lassalle and P.~Turco.
\newblock On null sequences for {B}anach operator ideals, trace duality and
  approximation properties.
\newblock {\em Mathematische Nachrichten}, 2017.

\bibitem{lindenstrauss1968absolutely}
J.~Lindenstrauss and A.~Pe{\l}czy{\'n}ski.
\newblock Absolutely summing operators in {$L_p$}-spaces and their
  applications.
\newblock {\em Studia Mathematica}, 29(3):275--326, 1968.

\bibitem{munoz2015alpha}
F.~Mu{\~n}oz, E.~Oja, and C.~Pi{\~n}eiro.
\newblock On $\alpha$-nuclear operators with applications to vector-valued
  function spaces.
\newblock {\em Journal of Functional Analysis}, 269(9):2871--2889, 2015.

\bibitem{Oikhberg-10}
T.~Oikhberg.
\newblock Completely bounded and ideal norms of multiplication operators and
  {S}chur multipliers.
\newblock {\em Integral Equations Operator Theory}, 66(3):425--440, 2010.

\bibitem{oja2012grothendieck}
E.~Oja.
\newblock Grothendieck's nuclear operator theorem revisited with an application
  to p-null sequences.
\newblock {\em Journal of Functional Analysis}, 263(9):2876--2892, 2012.

\bibitem{oja2012remark}
E.~Oja.
\newblock A remark on the approximation of $p$-compact operators by finite-rank
  operators.
\newblock {\em Journal of Mathematical Analysis and Applications},
  387(2):949--952, 2012.

\bibitem{pietsch1968ideale}
A.~Pietsch.
\newblock Ideale von operatoren in {B}anachr{\"a}umen.
\newblock {\em Mitt. Math. Gesselsch. DDR}, 1:1--14, 1968.

\bibitem{pietsch1978operator}
A.~Pietsch.
\newblock {\em Operator ideals}, volume~16.
\newblock Deutscher Verlag der Wissenschaften, 1978.

\bibitem{Pietsch-Operator-Ideals}
A.~Pietsch.
\newblock {\em Operator ideals}, volume~20 of {\em North-Holland Mathematical
  Library}.
\newblock North-Holland Publishing Co., Amsterdam, 1980.
\newblock Translated from German by the author.

\bibitem{pietsch2014ideal}
A.~Pietsch.
\newblock The ideal of $p$-compact operators and its maximal hull.
\newblock {\em Proceedings of the American Mathematical Society},
  142(2):519--530, 2014.

\bibitem{pineiro2011p}
C.~Pi{\~n}eiro and J.~M. Delgado.
\newblock $p$-convergent sequences and {B}anach spaces in which $p$-compact
  sets are $q$-compact.
\newblock {\em Proceedings of the American Mathematical Society},
  139(3):957--967, 2011.

\bibitem{Pisier-Asterisque-98}
G.~Pisier.
\newblock Non-commutative vector valued {$L_p$}-spaces and completely
  {$p$}-summing maps.
\newblock {\em Ast\'erisque}, (247):vi+131, 1998.

\bibitem{Pisier-Operator-Space-Theory}
G.~Pisier.
\newblock {\em Introduction to operator space theory}, volume 294 of {\em
  London Mathematical Society Lecture Note Series}.
\newblock Cambridge University Press, Cambridge, 2003.

\bibitem{reinov2000linear}
O.~Reinov.
\newblock On linear operators with $p$-nuclear adjoints.
\newblock {\em Vestn. SPbGU, Mat.}, 4(math. FA/0107113):24--27, 2000.

\bibitem{Ryan}
R.~A. Ryan.
\newblock {\em Introduction to tensor products of {B}anach spaces}.
\newblock Springer Monographs in Mathematics. Springer-Verlag London Ltd.,
  London, 2002.

\bibitem{sinha2002compact}
D.~P. Sinha and A.~K. Karn.
\newblock Compact operators whose adjoints factor through subspaces of lp.
\newblock {\em Studia Math}, 150(1):17--33, 2002.

\bibitem{prasad2008compact}
D.~P. Sinha and A.~K. Karn.
\newblock Compact operators which factor through subspaces of lp.
\newblock {\em Mathematische Nachrichten}, 281(3):412--423, 2008.

\bibitem{turco2016mathcal}
P.~Turco.
\newblock {$A$}-compact mappings.
\newblock {\em Revista de la Real Academia de Ciencias Exactas, F{\'\i}sicas y
  Naturales. Serie A. Matem{\'a}ticas}, 110(2):863--880, 2016.

\bibitem{webster1997local}
C.~Webster.
\newblock {\em Local operator spaces and applications}.
\newblock PhD thesis, University of California, 1997.

\bibitem{Webster1998}
C.~Webster.
\newblock Matrix compact sets and operator approximation properties.
\newblock {\em arXiv preprint math/9804093}, 1998.

\end{thebibliography}
\end{document}